\documentclass[final]{amsart}
\usepackage{graphicx}

\usepackage[numbers]{natbib}
\newcommand{\ym}[1]{}

\usepackage{amsmath}
\usepackage{amssymb}
\usepackage[all,cmtip]{xy}
\let\injlim\varinjlim
\let\projlim\varprojlim

\usepackage{enumitem}
\usepackage{hyperref}
\usepackage[capitalise,noabbrev]{cleveref}

\usepackage{ifdraft}
\ifdraft{\input{draft.tex}}{}

\crefname{enumi}{}{}
\creflabelformat{enumi}{#2(#1)#3}
\crefname{subsection}{Subsection}{Subsections}

\theoremstyle{plain}
\newtheorem{theorem*}{Theorem}
\expandafter\renewcommand\csname thetheorem*\endcsname{\Roman{theorem*}}
\newtheorem{theorem}{Theorem}[section]
\newtheorem{proposition}[theorem]{Proposition}
\newtheorem{lemma}[theorem]{Lemma}
\newtheorem{corollary}[theorem]{Corollary}
\theoremstyle{definition}
\newtheorem{definition}[theorem]{Definition}
\newtheorem{example}[theorem]{Example}
\newtheorem{question}[theorem]{Question}
\theoremstyle{remark}
\newtheorem{remark}[theorem]{Remark}

\newcommand{\NN}{\mathbf{N}}
\newcommand{\ZZ}{\mathbf{Z}}
\newcommand{\QQ}{\mathbf{Q}}
\newcommand{\RR}{\mathbf{R}}
\newcommand{\E}{\mathbb{E}}
\newcommand{\h}{\mathrm{h}}
\newcommand{\hyp}{\textnormal{hyp}}
\newcommand{\op}{\textnormal{op}}
\newcommand{\coh}{\textnormal{coh}}
\newcommand{\cons}{\textnormal{cons}}
\newcommand{\st}{\textnormal{st}}
\newcommand{\perf}{\textnormal{perf}}

\newcommand{\pt}{\operatorname{pt}}
\newcommand{\Alg}{\operatorname{Alg}}
\newcommand{\Alex}{\operatorname{Alex}}
\newcommand{\Arch}{\operatorname{Arch}}
\newcommand{\B}{\operatorname{B}}
\newcommand{\Fun}{\operatorname{Fun}}
\newcommand{\Idl}{\operatorname{Idl}}
\newcommand{\Ind}{\operatorname{Ind}}
\newcommand{\Free}{\operatorname{Free}}
\newcommand{\Hom}{\operatorname{Hom}}
\newcommand{\Map}{\operatorname{Map}}
\let\P\relax
\newcommand{\P}{\operatorname{P}}
\newcommand{\PShv}{\operatorname{PShv}}
\newcommand{\QCoh}{\operatorname{QCoh}}
\newcommand{\Shv}{\operatorname{Shv}}
\newcommand{\Spc}{\operatorname{Spc}}
\newcommand{\Spec}{\operatorname{Spec}}
\newcommand{\Zar}{\operatorname{Zar}}
\newcommand{\X}{\text{--}}
\newcommand{\sgeq}{\smash\geq}
\newcommand{\sngeq}{\smash\ngeq}
\newcommand{\unit}{\mathbf{1}}
\newcommand{\cat}[1]{\mathcal{#1}}
\newcommand{\Cat}[1]{\mathsf{#1}}

\let\autocite\cite

\title[TTG of filtered objects and sheaves]%
{Tensor triangular geometry of filtered objects and sheaves}
\author{Ko Aoki}
\address{
    Department of Mathematics,
    Tokyo Institute of Technology, 2--12--1 \=Ookayama, Meguro-ku,
    Tokyo 152--8551, Japan
}
\email{aoki.k.an@m.titech.ac.jp}
\date{\today}

\begin{document}

\begin{abstract}
    We compute
    the the Balmer spectra of compact objects of
    tensor triangulated categories
    whose objects are
    filtered or graded objects of
    (or sheaves valued in)
    another tensor triangulated category.
    Notable examples include the filtered derived category
    of a scheme as well as the homotopy category
    of filtered spectra.
    We use an $\infty$-categorical method to
    properly formulate and deal with the problem.
    Our computations are based
    on a point-free approach, so that
    distributive lattices and semilattices
    are used as key tools.

    In the appendix, we prove
    that the $\infty$-topos of hypercomplete sheaves
    on an $\infty$-site is recovered from a basis,
    which may be of independent interest.
\end{abstract}

\maketitle

\section{Introduction}\label{s_intro}

In the subject called tensor triangular geometry,
a basic object to study
is a tt-category,
which is a triangulated category equipped with a compatible
symmetric monoidal structure.
Balmer introduced a way to associate
to each tt-category~$ \cat{T}^{\otimes} $ a topological space
$ \Spc(\cat{T}) $, which we call the Balmer spectrum.
We refer the reader to \autocite{Balmer10}
for a survey on tensor triangular geometry.

In this paper, we compute the Balmer spectra
of mainly two families of tt-categories, whose objects are diagrams
in another tt-category.
Although those two computations are logically independent,
many techniques used
in them are similar.

To state the main results, we introduce some terminology.
See \cref{hikaku} for the comparison with
the common setting of tensor triangular geometry.

\begin{definition}\label{04732c0097}
    A \emph{big tt-$ \infty $-category} is a compactly generated stable
    $ \infty $-category equipped with an $ \E_2 $-monoidal
    structure whose tensor products preserve (small) colimits separately in each variable
    and restrict to compact objects.
\end{definition}

We here state only a main consequence of the first computation
because stating it in full generality requires some notions.

\begin{theorem*}\label{cbaf04a35f}
    Suppose that $ \cat{C}^{\otimes} $ is a big tt-$ \infty $-category.
    \begin{enumerate}
        \item
            \label{e6bd26145c}
            For a nonzero Archimedean group $ A $
            (for example, $ \ZZ $, $ \QQ $, or~$ \RR $, equipped with their
            usual orderings),
            there is a canonical homeomorphism
            \begin{equation*}
                \Spc(\Fun(A, \cat{C})^{\omega})
                \simeq
                S \times \Spc(\cat{C}^{\omega}),
            \end{equation*}
            where $ S $ denotes the Sierpi\'nski space
            (that is, the Zariski spectrum of a discrete valuation ring);
            see \cref{6b47fa8c88}.
        \item
            \label{04e366edb0}
            For an abelian group~$ A $,
            considered as a discrete symmetric monoidal poset,
            there is a canonical homeomorphism
            \begin{equation*}
                \Spc(\Fun(A, \cat{C})^{\omega})
                \simeq
                \Spc(\cat{C}^{\omega}).
            \end{equation*}
    \end{enumerate}
    In each statement, we consider the Day convolution
    $ \E_2 $-monoidal structure on the $ \infty $-category $ \Fun(A, \cat{C}) $.
\end{theorem*}

\begin{figure}[h]
    \includegraphics{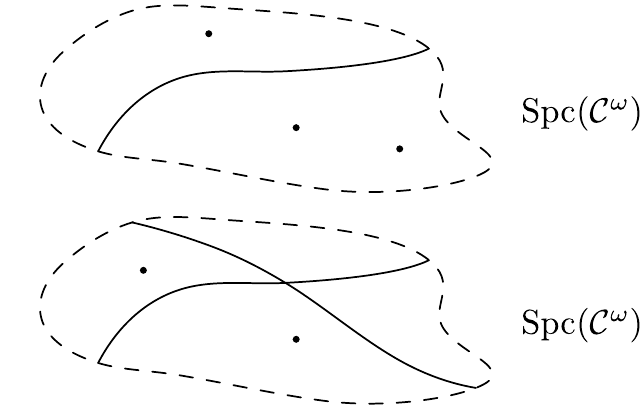}
    \caption{
        As a set, $ \Spc(\Fun(\ZZ, \cat{C})^{\omega}) $
        consists of two copies of $ \Spc(\cat{C}^{\omega}) $.
        Its closed subsets correspond to inclusions between
        two closed subsets of $ \Spc(\cat{C}^{\omega}) $.
    }
    \label{6b47fa8c88}
\end{figure}

Specializing to the case $ A = \ZZ $,
this theorem has several consequences:

\begin{example}\label{f083edff3f}
    For any quasicompact quasiseparated scheme~$ X $,
    the $ \infty $-category
    $ \QCoh(X) $, whose homotopy category is
    the derived category of discrete quasicoherent sheaves,
    becomes a big tt-$ \infty $-category
    when considering the (derived) tensor products.
    The reconstruction theorem (see \autocite[Theorem~54]{Balmer10}) implies that
    $ \Spc(\QCoh(X)^{\omega}) $ is the underlying topological space~$ X $.
    Applying our result to the case $ \cat{C}^{\otimes} = \QCoh(X)^{\otimes} $,
    we have that the Balmer spectrum of
    perfect filtered complexes on~$ X $ is the product of
    the Sierpi\'nski space
    and the underlying topological space of~$ X $.

    We note that in the special case
    where $ X $ is an affine scheme,
    this result is obtained by Gallauer in \autocite{Gallauer18}
    using a different method.
\end{example}

\begin{example}\label{71547b89c7}
    Since $ \Spc(\Cat{Sp}^{\omega}) $
    is already calculated (see \autocite[Theorem~51]{Balmer10}),
    we get the Balmer spectrum of (compact) filtered spectra.
\end{example}

\begin{example}\label{0543ef49d9}
    One advantage of the generality is
    that it can be applied iteratively.
    For example, it can be used
    when
    objects are $ \ZZ $-filtered in several directions.
    Actually, what we prove in \cref{s_21e9cca7d2}
    is so general that we can determine
    the Balmer spectrum of $ \ZZ^{\kappa} $-filtered objects
    for any cardinal~$ \kappa $.
\end{example}

We note that this theorem has a geometric interpretation:

\begin{example}\label{f15af58e5a}
    For an $ \E_{\infty} $-ring~$ R $,
    Moulinos proved in~\autocite{Moulinos} that
    there exist the following equivalences
    of symmetric monoidal $ \infty $-categories:
    \begin{align*}
        \Fun(\ZZ, \Cat{Mod}_R)^{\otimes}
        &\simeq \QCoh(\mathbb{A}^1_R/\mathbb{G}_{\mathrm{m},R})^{\otimes},
        &
        \Fun(\ZZ^{\textnormal{disc}}, \Cat{Mod}_R)^{\otimes}
        &\simeq \QCoh(\mathrm{B}\mathbb{G}_{\mathrm{m},R})^{\otimes},
    \end{align*}
    where $ \ZZ^{\textnormal{disc}} $ denotes an abelian group of integers
    without a poset structure.
    Applying our result
    to the case when
    $ A = \ZZ $,~$ \ZZ^{\textnormal{disc}} $ and
    $ \cat{C}^{\otimes} = \Cat{Mod}_R^{\otimes} $,
    we get
    the Balmer spectra of perfect complexes on
    these geometric stacks
    (but note that the determination of
    $ \Spc(\Cat{Mod}_R^{\omega}) $ for general $ R $
    is a difficult problem).
    At least
    when $ R $ is a field, our computations reflect
    a naive intuition on how these stacks look like.
\end{example}

The second result is the following:

\begin{theorem*}\label{43db5adc0d}
    Suppose that $ \cat{C}^{\otimes} $ is a big tt-$ \infty $-category
    and $ X $ is a coherent topological space
    (that is, a space which arises as the underlying
    topological space of a quasicompact quasiseparated scheme).
    Let $ \Shv_{\cat{C}}(X) $ denote the $ \infty $-category of $ \cat{C} $-valued
    sheaves on~$ X $.
    Considering the pointwise $ \E_2 $-monoidal structure on it,
    we have a canonical homeomorphism
    \begin{equation*}
        \Spc(\Shv_{\cat{C}}(X)^{\omega}) \simeq X_{\cons} \times \Spc(\cat{C}^{\omega}),
    \end{equation*}
    where $ X_{\cons} $ denotes
    the set~$ X $ endowed with the constructible topology of~$ X $.
\end{theorem*}

One big problem in tensor triangular geometry
is to compute
the Balmer spectrum of
compact objects of the stable motivic homotopy theory
$ \operatorname{SH}(X)^{\otimes} $ associated to
a quasicompact quasiseparated scheme~$ X $.
The first step would be the determination of
that of spectrum-valued sheaves
on the smooth-Nisnevich site of~$ X $.
\cref{43db5adc0d} can be considered as a toy case
of that calculation.
Nevertheless, it is an interesting fact in its own right.

In this paper,
we use distributive lattices to deal with the Balmer spectra.
More precisely,
we introduce the notion of the Zariski lattice of a tt-category,
which turns out to be just the opposite of the distributive lattice
of quasicompact open sets of the Balmer spectrum.
Thus it contains the same information as the Balmer spectrum,
but it has a more algebraic nature, which makes our computations possible.
We note that (upper) semilattices are also used in the proof.

\subsection*{Outline}

In \cref{s_94e3b0ab99},
we study $ \infty $-categorical machinery
(mainly) related to functor categories.
In \cref{s_7bfd3c8f11}, we review basic notions
in tensor triangular geometry using the language
of distributive lattices.
Its ``tensorless'' variant is also introduced there.
\cref{s_d6db2a7fad,s_21e9cca7d2} are devoted to
the proofs of \cref{43db5adc0d} and
a general version of \cref{cbaf04a35f}, respectively.
These two sections are logically independent,
but we arrange them in this way
because the former is quite simpler.
In \cref{s_basis},
we develop some technical material on $ \infty $-toposes
which we need in \cref{s_d6db2a7fad}.

\subsection*{Acknowledgments}

I would like to thank
Jacob Lurie for kindly answering several questions
on $ \infty $-categories.
I thank Martin Gallauer and Shane Kelly for commenting on
an early draft of this paper.

\section{Functor categories}\label{s_94e3b0ab99}

Concerning $ \infty $-categories,
we will follow terminology and notation used in
\autocite{HTT,HA,SAG} with minor exceptions,
which we will explicitly mention.

In this section, we study general properties
of the $ \infty $-category $ \Fun(K, \cat{C}) $
for a small $ \infty $-category~$ K $
and a presentable $ \infty $-category~$ \cat{C} $.

\subsection{Tensor products of presentable $\infty$-categories}\label{ss_2b21e083fa}

Let $ \Cat{Pr} $ denote the very large $ \infty $-category
of presentable $ \infty $-categories whose morphisms
are those functors that preserve colimits.
This is what is denoted by $ \Cat{Pr}^{\mathrm{L}} $ in~\autocite{HTT}.
Similarly, for an infinite regular cardinal $ \kappa $,
we let $ \Cat{Pr}_{\kappa} $ denote
the very large $ \infty $-category
of $ \kappa $-compactly generated $ \infty $-categories whose morphisms
are those functors that preserve colimits and $ \kappa $-compact objects.

Here we list basic properties of the symmetric monoidal structure on~$ \Cat{Pr} $
as one theorem; see \autocite[Section~4.8.1]{HA} for the proofs.

\begin{theorem}\label{tensor}
    There exists a symmetric monoidal structure on~$ \Cat{Pr} $
    which satisfies the following properties:
    \begin{enumerate}
        \item
            \label{1809666cc6}
            For $ \cat{C}, \cat{D} \in \Cat{Pr} $,
            we have that $ \cat{C} \otimes \cat{D} $ is canonically
            equivalent to the full subcategory
            of $ \Fun(\cat{C}^{\op}, \cat{D}) $ spanned
            by those functors preserving limits.
        \item
            \label{9440ef3338}
            For $ \cat{C}_1, \dotsc, \cat{C}_n, \cat{D} \in \Cat{Pr} $,
            the full subcategory of
            $ \Fun(\cat{C}_1 \otimes \dotsb \otimes \cat{C}_n, \cat{D}) $
            spanned by those functors preserving colimits
            is equivalent to that
            of $ \Fun(\cat{C}_1 \times \dotsb \times \cat{C}_n, \cat{D}) $
            spanned by those functors preserving colimits in each variable.
        \item
            \label{b1b979b395}
            The functor
            $ {\PShv} \colon \Cat{Cat}_{\infty} \to \Cat{Pr} $
            has a symmetric monoidal refinement,
            where the large $ \infty $-category $ \Cat{Cat}_{\infty} $
            of $ \infty $-categories
            is considered to be equipped with the cartesian
            symmetric monoidal structure.
        \item
            \label{ffd2ebeeb5}
            The tensor product operations
            preserve small colimits in each variable.
        \item
            \label{dee93321da}
            For an infinite regular cardinal~$ \kappa $
            and $ \cat{C}_1, \dotsc, \cat{C}_n, \cat{D} \in \Cat{Pr}_{\kappa} $,
            we have $ \cat{C}_1 \otimes \dotsb \otimes \cat{C}_n \in \Cat{Pr}_{\kappa} $.
            Moreover,
            the full subcategory of
            $ \Fun(\cat{C}_1 \otimes \dotsb \otimes \cat{C}_n, \cat{D}) $
            spanned by those functors preserving colimits and $ \kappa $-compact objects
            is canonically equivalent to that
            of $ \Fun(\cat{C}_1 \times \dotsb \times \cat{C}_n, \cat{D}) $
            spanned by those functors
            that preserve colimits in each variable and
            restrict to determine functors
            from
            $ (\cat{C}_1)^{\kappa} \times \dotsb \times (\cat{C}_n)^{\kappa} $
            to~$ \cat{D}^{\kappa} $.
    \end{enumerate}
\end{theorem}

The useful consequence for us is the following:

\begin{corollary}\label{188b761524}
    For a presentable $ \infty $-category~$ \cat{C} $
    and a small $ \infty $-category~$ K $,
    we have a canonical equivalence
    $ \Fun(K, \Cat{S}) \otimes \cat{C} \simeq \Fun(K, \cat{C}) $.
\end{corollary}

\begin{proof}
    According to~\cref{1809666cc6} of \cref{tensor},
    the left hand side can be regarded as
    the full subcategory
    of $ \Fun(\Fun(K, \Cat{S})^{\op}, \cat{C}) $
    spanned by those functors preserving limits.
    Hence the equivalence follows from \autocite[Theorem~5.1.5.6]{HTT}.
\end{proof}

We have the following result from this description,
although it can be proven more concretely:

\begin{corollary}\label{bd6b741a6f}
    Let $ \kappa $ be an infinite regular cardinal.
    If $ \cat{C} $ is a $ \kappa $-compactly generated $ \infty $-category,
    so is $ \Fun(K, \cat{C}) $
    for any small $ \infty $-category~$ K $.
\end{corollary}

\subsection{Compact objects in a functor category}\label{ss_6fb1913383}

In this subsection, we fix an infinite regular cardinal~$ \kappa $.

For a $ \kappa $-compactly generated $ \infty $-category $ \cat{C} $,
according to \cref{bd6b741a6f},
the $ \infty $-category $ \Fun(K, \cat{C}) $ is
also
$ \kappa $-compactly generated for any small $ \infty $-category $ K $.
The aim of this subsection is
to determine $ \kappa $-compact objects
of $ \Fun(K, \cat{C}) $ under some assumptions on~$ K $.

\begin{definition}\label{c2d2794786}
    Let $ K $ be a small $ \infty $-category.
    We say that $ K $ is \emph{$ \kappa $-small} if
    there exist a simplicial set $ K' $ with $ < \kappa $ nondegenerate simplices
    and a Joyal equivalence $ K' \to K $.
    We say that $ K $ is \emph{locally $ \kappa $-compact} if
    the mapping space functor $ {\Map} \colon K^{\op} \times K \to \Cat{S} $ factors
    through the full subcategory $ \Cat{S}^{\kappa} $ spanned by $ \kappa $-compact spaces.
\end{definition}

\begin{remark}\label{fbcd64e8ad}
    When $ \kappa $ is uncountable,
    the notion of $ \kappa $-smallness introduced here
    coincides with that of essential $ \kappa $-smallness
    given in \autocite[Definition~5.4.1.3]{HTT}.
    According to \autocite[Proposition~5.4.1.2]{HTT},
    every $ \kappa $-small category is locally $ \kappa $-compact.
    However in the case $ \kappa = \omega $,
    the analogous result does not hold:
    For example, the $ 1 $-sphere~$ S^1 $,
    regarded as an $ \infty $-category,
    is $ \omega $-finite but not locally $ \omega $-compact;
    see also \cref{b80336b02b}.
\end{remark}

\begin{lemma}\label{8fe6e0f750}
    Let
    $ \cat{C} $ be a $ \kappa $-compactly generated $ \infty $-category
    and $ K $ a $ \kappa $-small $ \infty $-category.
    Then every functor $ K \to \cat{C} $
    that factors through~$ \cat{C}^{\kappa} $
    is a $ \kappa $-compact object of $ \Fun(K, \cat{C}) $.
\end{lemma}

\begin{proof}
    This is a corollary of \autocite[Proposition~5.3.4.13]{HTT}.
\end{proof}

\begin{lemma}\label{3a0098a7a8}
    Let
    $ \cat{C} $ be a $ \kappa $-compactly generated $ \infty $-category
    and $ K $ a locally $ \kappa $-compact $ \infty $-category.
    Then every $ \kappa $-compact object $ K \to \cat{C} $ of $ \Fun(K, \cat{C}) $
    factors through~$ \cat{C}^{\kappa} $.
\end{lemma}

\begin{proof}
    Let $ k $ be an object of~$ K $.
    The inclusion $ i \colon {*} \hookrightarrow K $ corresponding to~$ k $
    induces a functor $ i^* \colon \Fun(K, \mathcal{C}) \to \mathcal{C} $
    by composition.
    Since $ F(k) \simeq i^*F $ holds for every functor $ F \colon K \to \mathcal{C} $,
    it suffices to show that the functor~$ i^* $ preserves $ \kappa $-compact objects.
    By \autocite[Proposition~5.5.7.2]{HTT},
    this is equivalent to the assertion
    that its right adjoint~$ i_* $ preserves $ \kappa $-filtered colimits.
    The functor~$ i_* $ can be concretely described
    as the assignment $ C \mapsto \bigl(l \mapsto C^{\Map_K(l,k)}\bigr) $
    using the cotensor structure on~$ \mathcal{C} $.
    Since $ \cat{C} $
    is $ \kappa $-compactly generated
    and $ \Map_K(l,k) $ is $ \kappa $-compact for every object
    $ l \in K $,
    we can see that $ i_* $ preserves $ \kappa $-filtered colimits.
\end{proof}

Combining these two lemmas,
we get the following result:

\begin{proposition}\label{1f1e4383ed}
    Let
    $ \cat{C} $ be a $ \kappa $-compactly generated $ \infty $-category
    and $ K $ a $ \kappa $-small locally $ \kappa $-compact $ \infty $-category.
    Then an object $ F \in \Fun(K, \cat{C}) $
    is $ \kappa $-compact if and only if
    it takes values in $ \kappa $-compact objects.
\end{proposition}

\begin{remark}\label{b80336b02b}
    \Cref{1f1e4383ed} does not hold without the local condition:
    Take $ K = S^1 $ and
    consider the object $ X $ of $ \Fun(S^1, \Cat{S}) $ corresponding to
    the universal covering $ {*} \to S^1 $ by the Grothendieck construction.
    By \autocite[Lemma~5.1.6.7]{HTT}
    the object~$ X $ is compact,
    but $ X({*}) \simeq \ZZ $ is not compact.
\end{remark}

\begin{remark}\label{8f4cd6fe0e}
    In the case $ \kappa = \omega $,
    the assumption on~$ K $ in the statement of \Cref{1f1e4383ed} is
    restrictive.
    For example,
    if $ K $ is also assumed to be equivalent to a space,
    $ K $ must be equivalent to a finite set:
    When $ K $ is simply connected,
    this can be deduced by considering the homological Serre
    spectral sequence associated to
    the fiber sequence $ \Omega K \to {*} \to K $.
    The general case follows by taking the universal cover of
    each connected component of~$ K $
    and using the fact that the classifying space
    of a nontrivial finite group is not finite.
\end{remark}

Later in this paper,
we consider a slightly more general situation
than that of \cref{1f1e4383ed}.
At that moment, the following result is useful:

\begin{corollary}\label{02eb77190d}
    Let $ K $ be a locally $ \kappa $-compact
    $ \infty $-category
    and $ \cat{C} $ a $ \kappa $-compactly generated $ \infty $-category.
    Suppose that
    there is a $ \kappa $-directed family $ (K_j)_{j \in J} $
    of full subcategories of~$ K $
    such that $ K_j $ is $ \kappa $-small for $ j \in J $.
    We let $ (i_j)_! $ denote the left Kan extension functor
    along the inclusion $ i_j \colon K_j \hookrightarrow K $.
    Then we have an equality
    \begin{equation*}
        \Fun(K, \cat{C})^{\kappa}
        = \bigcup_{j \in J} (i_j)_! (\Fun(K_j, \cat{C})^{\kappa})
    \end{equation*}
    of full subcategories of $ \Fun(K, \cat{C}) $.
\end{corollary}

\begin{proof}
    Since the right adjoint of
    the left Kan extension functor preserves colimits,
    we obtain one inclusion by applying \autocite[Proposition~5.5.7.2]{HTT}.

    Conversely,
    let $ F \in \Fun(K, \cat{C}) $ be a $ \kappa $-compact object.
    By assumption, $ F $ is the colimit of
    the $ \kappa $-filtered diagram
    $ j \mapsto (i_j)_!(F \rvert_{K_j}) $. Hence we can take $ j \in J $
    such that $ F $ is a retract of $ (i_j)_!(F \rvert_{K_j}) $.
    Since the essential image of $ (i_j)_! $ is closed under retracts,
    $ F $ is in fact equivalent to $ (i_j)_!(F \rvert_{K_j}) $.
    From \cref{3a0098a7a8,1f1e4383ed}, we have that
    $ F \rvert_{K_j} $ is $ \kappa $-compact,
    which completes the proof of the other inclusion.
\end{proof}

\subsection{Recollement}\label{ss_c6b04c731b}

We refer the reader to \autocite[Section~A.8]{HA}
for the theory of recollement for $ \infty $-categories.

\begin{definition}\label{7c03bb65b4}
    Suppose that $ \cat{C} $ is an $ \infty $-category
    and $ i \colon \cat{C}_0 \to \cat{C} $
    and $ j \colon \cat{C}_1 \to \cat{C} $ are fully faithful functors.
    We say that $ \cat{C} $ is a \emph{recollement} 
    of~$ i $ and~$ j $ if
    $ \cat{C} $ is a recollement of
    the essential images of~$ i $ and~$ j $
    in the sense of \autocite[Definition~A.8.1]{HA}.
\end{definition}

There are many ways to write
a functor category as a recollement due to the following observation:

\begin{proposition}\label{883b124c64}
    Let $ \cat{C} $ be a presentable $ \infty $-category
    and $ K_1 \subset K $ a full inclusion of $ \infty $-categories.
    Suppose that $ K_1 $ is a cosieve on~$ K $;
    that is, $ k \in K_1 $ implies $ l \in K_1 $
    if there exists a morphism $ k \to l $ in~$ K $.
    Let $ K_0 $ denote the full subcategory of~$ K $
    spanned by objects not in~$ K_1 $.
    Let $ i_* \colon \Fun(K_0, \cat{C}) \hookrightarrow \Fun(K, \cat{C}) $
    and
    $ j_* \colon \Fun(K_1, \cat{C}) \hookrightarrow \Fun(K, \cat{C}) $
    denote the functors defined by right Kan extensions.
    Then $ \Fun(K, \cat{C}) $ is a recollement of~$ i_* $ and~$ j_* $.
\end{proposition}

\begin{proof}
    The only nontrivial point is
    to verify
    that $ j^*i_* $ sends every object to the final object,
    where we write $ j^* $ for the functor
    given by restriction along the inclusion $ K_1 \subset K $.
    For $ k \in K_1 $,
    the cosieve condition implies $ (K_0)_{k/} = \emptyset $.
    Hence for any $ F \in \Fun(K_0, \cat{C}) $ and any $ k \in K_1 $, we have
    $ (j^*i_*F)(k) = (i_*F)(k) \simeq \projlim_{l \in (K_0)_{k/}} F(l) \simeq {*} $,
    which completes the proof.
\end{proof}

We state a lemma on recollements
in the stable setting.

\begin{lemma}\label{recol-spl}
    Let $ \cat{C} $ be a stable $ \infty $-category,
    which is a recollement
    of $ i_* $ and $ j_* $.
    We write $ i^* $, $ j^* $
    for the left adjoints of~$ i_* $, $ j_* $, respectively,
    and $ j_! $ for that of~$ j^* $.
    Then for an object $ C \in \cat{C} $,
    the cofiber sequence $ j_!j^*C \to C \to i_*i^*C $
    splits if and only if the map
    $ i^*C \to i^*j_*j^*C $ is zero.
\end{lemma}

\begin{proof}
    We write $ i^! $ for the right adjoint of~$ i_* $.
    The ``only if'' direction follows from
    the fact that $ i^*(j_!j^*C) $ and $ i^*j_*j^*(i_*i^*C) $ are both zero.
    We wish to prove the converse.
    By applying $ i^* $ to the cofiber sequence
    $ i_*i^!C \to C \to j_*j^*C $ and shifting,
    we obtain a cofiber sequence
    $ \Sigma^{-1}i^*j_*j^*C \to i^!C \to i^*C $,
    which splits by assumption.
    Then the map $ i_*i^*C \to C $ corresponding
    to the section $ i^*C \to i^!C $ by adjunction
    induces the desired splitting.
\end{proof}

\subsection{The two monoidal structures on a functor category}\label{ss_611af07fdc}

In this subsection, we fix an $ \infty $-operad $ \E_k^{\otimes} $,
where $ k $ is a positive integer or the symbol~$ \infty $.

Recall that an $ \E_k $-monoidal $ \infty $-category
can be regarded as an $ \E_k $-algebra object of
the symmetric monoidal category $ \Cat{Cat}_{\infty}^{\times} $.
In the following discussion,
we often use this identification implicitly.

The following is a special case of \autocite[Proposition~3.2.4.4]{HA}:

\begin{lemma}\label{8e9b197f4f}
    There exist a symmetric monoidal structure on $ \Alg_{\E_k}(\Cat{Pr}) $
    such that
    the forgetful functor $ \Alg_{\E_k}(\Cat{Pr}) \to \Cat{Pr} $
    has a symmetric monoidal refinement.
    Also, the same holds for
    $ \Cat{Pr}_{\kappa}^{\otimes} $, where $ \kappa $ is an infinite regular cardinal.
\end{lemma}

First we use this to construct the pointwise $ \E_k $-monoidal structure.

\begin{definition}\label{a023043790}
    Let $ K $ be a small $ \infty $-category
    and $ \cat{C}^{\otimes} $ a presentable $ \E_k $-monoidal $ \infty $-category
    whose tensor products preserve colimits in each variable.
    Then combining with the cartesian $ \E_k $-monoidal structure on $ \Fun(K, \Cat{S}) $,
    we obtain an $ \E_k $-monoidal structure on
    $ \Fun(K, \cat{C}) \simeq \Fun(K, \Cat{S}) \otimes \cat{C} $ by
    using \cref{8e9b197f4f}.
    We call this the \emph{pointwise $\E_k$-monoidal structure}.
\end{definition}

The pointwise tensor products can be computed pointwise, as the name suggests.

We consider a condition under which this construction is compatible
with the compact generation property.
See also \cref{3752cd1622}
for another result in this direction.

\begin{proposition}\label{5a932f09f2}
    Let $ \kappa $ be an infinite regular cardinal.
    In the situation of \cref{a023043790},
    suppose furthermore that $ K $ is $ \kappa $-small and locally $ \kappa $-compact,
    $ \cat{C} $ is $ \kappa $-compactly generated,
    and the tensor products on $ \cat{C} $ restrict to~$ \cat{C}^{\kappa} $.
    Then the pointwise tensor products on $ \Fun(K, \cat{C}) $
    also restrict to $ \Fun(K, \cat{C})^{\kappa} $.
\end{proposition}

\begin{proof}
    Without loss of generality we may assume that $ \cat{C} = \Cat{S} $.
    Since finite products of $ \kappa $-compact spaces are again $ \kappa $-compact,
    the desired result follows from \cref{1f1e4383ed}.
\end{proof}

We then consider the Day convolution $ \E_k $-monoidal structure.
We first note that for an $ \E_k $-monoidal $ \infty $-category~$ K^{\otimes} $,
we can equip a canonical $ \E_k $-monoidal structure
on the opposite~$ K^{\op} $; see \autocite[Remark~2.4.2.7]{HA}.
Therefore, for such $ K^{\otimes} $,
we have an $ \E_k $-monoidal structure on $ \Fun(K, \Cat{S})
\simeq \PShv(K^{\op}) $
by~\cref{b1b979b395} of \cref{tensor}.

\begin{definition}\label{c3c54673e6}
    Let $ K^{\otimes} $ be a small $ \E_k $-monoidal $ \infty $-category
    and $ \cat{C}^{\otimes} $ a presentable $ \E_k $-monoidal $ \infty $-category
    whose tensor products preserve colimits in each variable.
    Then considering the $ \E_k $-monoidal structure on $ \Fun(K, \Cat{S}) $
    explained above,
    we obtain an $ \E_k $-monoidal structure on
    $ \Fun(K, \cat{C}) \simeq \Fun(K, \Cat{S}) \otimes \cat{C} $ by
    using \cref{8e9b197f4f}.
    We call this the \emph{Day convolution $\E_k$-monoidal structure}.
\end{definition}

Concretely,
the Day convolution tensor products can be computed as follows:

\begin{lemma}\label{9d1faa2fd7}
    In the situation of \cref{c3c54673e6},
    the Day convolution tensor product
    of $ F_1, \dotsc, F_n \in \Fun(K, \cat{C}) $
    is equivalent to the left Kan extension of the composite
    \begin{equation*}
        K \times \dotsb \times K \xrightarrow{F_1 \times \dotsb \times F_n}
        \cat{C} \times \dotsb \times \cat{C}
        \xrightarrow{\otimes} \cat{C}
    \end{equation*}
    along the tensor product $ K \times \dotsb \times K \to K $.
    Hence
    for $ k \in K $ we have
    \begin{equation*}
        (F_1 \otimes \dotsb \otimes F_n)(k)
        \simeq \injlim_{k_1 \otimes \dotsb \otimes k_n \to k}
        F_1(k_1) \otimes \dotsb \otimes F_n(k_n).
    \end{equation*}
\end{lemma}

We note that in the case $ \cat{C} = \Cat{S} $,
this is claimed in \autocite[Remark~4.8.1.13]{HA}.

\begin{proof}
    By universality,
    we have a canonical map from
    the functor constructed in the statement
    to the tensor product.
    Since both constructions are compatible with colimits in each variable,
    we can assume that $ \cat{C} = \Cat{S} $
    and that $ F_1 $, \dots,~$ F_n $ are in the image of the Yoneda embedding
    $ K^{\op} \hookrightarrow \PShv(K^{\op}) $.
    In this case, the desired claim is trivial.
\end{proof}

The author learned the following fact from Jacob Lurie,
which says that the Day convolution counterpart of \cref{5a932f09f2}
does not need any assumption on~$ K^{\otimes} $:

\begin{lemma}\label{2b52e3227f}
    Let $ \kappa $ be an infinite regular cardinal.
    In the situation of \cref{c3c54673e6},
    suppose furthermore that
    $ \cat{C} $ is $ \kappa $-compactly generated
    and the tensor products on~$ \cat{C} $ restrict to $ \cat{C}^{\kappa} $.
    Then the Day convolution tensor products on $ \Fun(K, \cat{C}) $
    also restrict to $ \Fun(K, \cat{C})^{\kappa} $.
\end{lemma}

\begin{proof}
    Without loss of generality we may assume that $ \cat{C} = \Cat{S} $.
    According to \autocite[Proposition~5.3.4.17]{HTT}, we have that
    $ \PShv(K^{\op})^{\kappa} $ is the smallest full subcategory
    of $ \PShv(K^{\op}) $ that contains
    the image of the Yoneda embedding and is closed under
    $ \kappa $-small colimits and retracts.
    Since the Yoneda embedding has an $ \E_k $-monoidal refinement,
    the desired result follows.
\end{proof}

Now we give a comparison result of these two $ \E_k $-monoidal structures.

\begin{proposition}\label{e371b27143}
    In the situation of \cref{c3c54673e6},
    suppose furthermore that the $ \E_k $-monoidal structure
    on~$ K $ is cocartesian.
    Then on $ \Fun(K, \cat{C}) $ the pointwise
    and Day convolution $ \E_k $-monoidal structures are equivalent.
\end{proposition}

\begin{proof}
    Without loss of generality we may assume that $ \cat{C} = \Cat{S} $
    and $ k = \infty $.
    Since both tensor products preserve colimits in each variable
    and restrict the image of the Yoneda embedding,
    it suffices to show that
    they are equivalent on the image.
    Hence the result follows from
    the uniqueness of cartesian symmetric monoidal structures on~$ K^{\op} $.
\end{proof}

Combining this with \cref{2b52e3227f},
we have the following:

\begin{corollary}\label{3752cd1622}
    Let $ \kappa $ be an infinite regular cardinal.
    In the situation of \cref{a023043790},
    suppose furthermore that
    $ K $ has finite coproducts and
    $ \cat{C} $ is $ \kappa $-compactly generated.
    Then the conclusion of \cref{5a932f09f2} holds.
\end{corollary}

\begin{remark}\label{7891bef018}
    We cannot completely remove assumptions on~$ K $:
    When $ K = \ZZ $, the final object of $ \Fun(\ZZ, \Cat{S}) $
    is not compact.
    See also \cref{31116799a2}.
\end{remark}

\section{Latticial approach to tensor triangular geometry}\label{s_7bfd3c8f11}

In this section, first we review the notion of
the Balmer spectrum using distributive lattices.
In \cref{s_tensorless} we introduce a variant of
tensor triangular geometry.

\subsection{Our setting}\label{ss_b7e8dab4b7}

Let $ \Cat{Pr}_{\omega}^{\st} $ denote
the full subcategory of $ \Cat{Pr}_{\omega} $ spanned
by compactly generated stable $ \infty $-categories,
to which
the symmetric monoidal structure on~$ \Cat{Pr}_{\omega} $
explained in \cref{tensor} restricts.
A big tt-$ \infty $-category, which is defined in \cref{04732c0097},
can be seen as
an $ \E_2 $-algebra object of $ (\Cat{Pr}_{\omega}^{\st})^{\otimes} $.

We let $ \Cat{Cat}_{\infty}^{\perf} $ denote the large
$ \infty $-category of idempotent complete stable $ \infty $-categories
whose morphisms are exact functors.
The equivalence
$ {\Ind} \colon \Cat{Cat}_{\infty}^{\perf} \to \Cat{Pr}_{\omega}^{\st} $
induces a symmetric monoidal structure on it.

\begin{definition}\label{016acc2052}
    A \emph{tt-$ \infty $-category}
    is an $ \E_2 $-algebra object of
    $ (\Cat{Cat}_{\infty}^{\perf})^{\otimes} $;
    in concrete terms, a tt-$ \infty $-category is an idempotent complete
    stable $ \infty $-category equipped with an $ \E_2 $-monoidal
    structure whose tensor products are exact in each variable.
\end{definition}

By definition,
the large $ \infty $-category $ \Alg_{\E_2}(\Cat{Cat}_{\infty}^{\perf}) $
of tt-$ \infty $-categories
and the very large $ \infty $-category $ \Alg_{\E_2}(\Cat{Pr}_{\omega}^{\st}) $
of big tt-$ \infty $-categories are equivalent.

\begin{remark}\label{hikaku}
    There are several differences between our setting
    and that of the usual theory of tensor triangular geometry,
    as found in \autocite[Hypothesis~21]{Balmer10}:
    \begin{enumerate}
        \item
            We use an $ \infty $-categorical enhancement.
            We note that by \autocite[Lemma~1.2.4.6]{HA}
            the idempotent completeness assumptions in both settings
            are equivalent.
        \item
            We consider an $ \E_2 $-monoidal structure,
            so that the induced monoidal structure on the underlying
            triangulated category is not necessarily symmetric,
            but braided.
            Actually,
            the arguments of this paper works
            with slight modifications
            even if calling an ($ \E_1 $-)algebra object of
            $ (\Cat{Cat}_{\infty}^{\perf})^{\otimes} $ a tt-$ \infty $-category,
            mainly since many notions including the Balmer spectrum
            only depend on the underlying ($ \E_1 $-)monoidal structure.
            However, the author does not know if such a generalization
            is useful.
        \item
            We do not impose any rigidity condition.
            This is because we do not need it
            for our computations.
    \end{enumerate}
\end{remark}

\subsection{The Stone duality}\label{ss_4ad569a7b6}

We review the Stone duality for distributive lattices.
For the basic theory, we refer the reader to \autocite[Section~II.3]{Johnstone82}.

We write $ \Cat{DLat} $, $ \Cat{Loc} $, $ \Cat{Loc}^{\coh} $, $ \Cat{Top} $
and $ \Cat{Top}^{\coh} $ for
the category of distributive lattices, locales, coherent locales,
topological spaces, and coherent topological spaces (also called spectral spaces),
respectively.
Note that
both inclusions
$ \Cat{Loc}^{\coh} \subset \Cat{Loc} $ and $ \Cat{Top}^{\coh} \subset \Cat{Top} $
are not full since only quasicompact maps are considered as morphisms in them.
The Stone duality for distributive lattices states that
the ideal frame functor
$ {\Idl} \colon \Cat{DLat} \to (\Cat{Loc}^{\coh})^{\op} $ and
the spectrum functor
$ {\Spec} \colon \Cat{DLat}^{\op} \to \Cat{Top}^{\coh} $ are equivalences,
and these two are compatible with the functor
$ {\pt} \colon \Cat{Loc} \to \Cat{Top} $
that sends a locale to its space of points.

We have the following consequences:

\begin{lemma}\label{6d78f23da1}
    The (nonfull) inclusion $ \Cat{Loc}^{\coh} \hookrightarrow \Cat{Loc} $ preserves
    (small) limits.
\end{lemma}

\begin{proof}
    This follows from 
    \autocite[Corollary~II.2.11]{Johnstone82}
    and the Stone duality.
\end{proof}

\begin{proposition}\label{5f0590fbb0}
    The spectrum functor
    $ \Cat{DLat}^{\op} \to \Cat{Top} $ preserves
    (small) limits.
\end{proposition}

\begin{proof}
    Since the functor $ {\pt} \colon \Cat{Loc} \to \Cat{Top} $
    has a left adjoint, which sends a topological space to its frame of open sets,
    it preserves limits.
    From \cref{6d78f23da1} and the fact that coherent locales are spatial,
    we obtain the result.
\end{proof}

\subsection{The Zariski lattice and the Balmer spectrum}\label{ss_d4fc9308b6}

In this subsection, we introduce the notion of the Zariski lattice
of a tt-$ \infty $-category.

\begin{definition}\label{8929dee0e0}
    Let $ \cat{T}^{\otimes} $ be a tt-$ \infty $-category.
    A \emph{radical ideal} of~$ \cat{T}^{\otimes} $ is
    a stable full replete subcategory
    $ I \subset \cat{T} $ that satisfies the following conditions:
    \begin{enumerate}
        \item
            \label{9d9ed4c715}
            If $ C \oplus D \in I $ for some $ C, D \in \cat{T} $,
            we have $ C, D \in I $.
        \item
            \label{3fcd62e277}
            For any $ C \in \cat{T} $ and $ D \in I $, we have $ C \otimes D \in I $.
        \item
            \label{f0e5bd7e43}
            If $ C \in \cat{T} $ satisfies
            $ C^{\otimes k} \in I $ for some $ k \geq 0 $, we have $ C \in I $.
    \end{enumerate}
    We denote the smallest radical ideal containing an object $ C \in \cat{T} $
    by~$ \sqrt{C} $.
\end{definition}

We note that the notion of
radical ideal of a tt-$ \infty $-category~$ \cat{T}^{\otimes} $
only depends on the underlying tensor triangulated category~$ (\h\cat{T})^{\otimes} $.

\begin{definition}\label{c167a286d1}
    Let $ \cat{T}^{\otimes} $ be a tt-$ \infty $-category.
    A \emph{support} for~$ \cat{T}^{\otimes} $ is a pair $ (L,s) $ of
    a distributive lattice~$ L $ and a function $ s \colon \cat{T} \to L $
    satisfying the following:
    \begin{enumerate}[start=0]
        \item
            \label{064da024a5}
            The function $ s $ takes the same values
            on equivalent objects.
            Hence we can evaluate $ s(C) $ even if $ C $
            is only determined up to equivalence.
        \item
            \label{efc38747ae}
            For $ C_1, \dotsc, C_n \in \cat{T} $, we have $
                s(C_1 \oplus \dotsb \oplus C_n)
                = s(C_1) \vee \dotsb \vee s(C_n)
            $.
            In particular, we have $ s(0) = 0 $.
        \item
            \label{bd772585de}
            For any cofiber sequence $ C' \to C \to C'' $ in~$ \cat{T} $,
            we have $ s(C') \vee s(C) = s(C) \vee s(C'') = s(C'') \vee s(C') $.
        \item
            \label{dbba1d4dc6}
            For $ C_1, \dotsc, C_n \in \cat{T} $, we have $
                s(C_1 \otimes \dotsb \otimes C_n)
                = s(C_1) \wedge \dotsb \wedge s(C_n)
            $.
            In particular, we have $ s(\unit) = 1 $,
            where $ \unit $ denotes the unit.
    \end{enumerate}
    They form a category,
    with morphisms $ (L,s) \to (L',s') $
    defined to be morphisms
    of distributive lattices $ f \colon L \to L' $ satisfying
    $ f \circ s = s' $.
\end{definition}

\begin{remark}\label{29b9ac4642}
    The notion of support introduced here
    is different from what is called ``support on~$ \cat{T} $''
    in \autocite[Definition~3.2.1]{KockPitsch17}, which values in a frame.
\end{remark}

\begin{definition}\label{88cd5e7146}
    The \emph{Zariski lattice} $ \Zar(\cat{T}) $
    of a tt-$ \infty $-category~$ \cat{T}^{\otimes} $ is
    the partially ordered set
    $ \bigl\{\sqrt{C} \mathrel{\bigl|} C \in \cat{T}\bigr\} $
    ordered by inclusion.
\end{definition}

\begin{proposition}\label{1745d1eb17}
    For a tt-$ \infty $-category $ \cat{T}^{\otimes} $,
    the following assertions hold:
    \begin{enumerate}
        \item
            \label{9061e8b182}
            The Zariski lattice $ \Zar(\cat{T}) $ is a distributive lattice.
        \item
            \label{9cf46b828c}
            The pair $ \bigl(\Zar(\cat{T}), C \mapsto \sqrt{C}\bigr) $
            is a support for~$ \cat{T}^{\otimes} $.
        \item
            \label{d7fcfff637}
            It is an initial support; that is,
            an initial object of the category described in \cref{c167a286d1}.
    \end{enumerate}
\end{proposition}

Although this can be proven directly,
here we give a proof using
several results obtained in \autocite[Section~3]{KockPitsch17}.

\begin{proof}
    According to \autocite[Theorem~3.1.9]{KockPitsch17},
    all radical ideals of~$ \cat{T}^{\otimes} $ form a coherent frame
    by inclusion
    and its compact objects are precisely the elements of $ \Zar(\cat{T}) $,
    so \cref{9061e8b182} holds.
    Also, \cref{9cf46b828c} follows from this observation,
    together with \autocite[Lemma~3.2.2]{KockPitsch17}.
    Assertion \cref{d7fcfff637} follows
    from essentially the same argument as that
    given in the proof of \autocite[Theorem~3.2.3]{KockPitsch17}.
\end{proof}

\begin{remark}\label{05648f71d8}
    This construction determines a functor
    $ {\Zar} \colon \Alg_{\E_2}(\Cat{Cat}_{\infty}^{\perf}) \to \Cat{DLat} $.
\end{remark}

Now we can give a definition
of the Balmer spectrum in this paper;
this is equivalent to the original definition
by \autocite[Theorem~3.1.9 and Corollary~3.4.2]{KockPitsch17}.

\begin{definition}\label{41f1f4f38a}
    For a tt-$ \infty $-category~$ \cat{T}^{\otimes} $,
    we let $ \Spc(\cat{T}) $ denote
    the coherent topological space $ \Spec(\Zar(\cat{T})^{\op}) $
    and call it the \emph{Balmer spectrum} of~$ \cat{T}^{\otimes} $.
\end{definition}

We conclude this subsection by
proving a property of the functor~$ \Zar $.

\begin{lemma}\label{continuity}
    The functor $ \Zar \colon
    \Alg_{\E_2}(\Cat{Cat}_{\infty}^{\perf}) \to \Cat{DLat} $
    preserves filtered colimits.
\end{lemma}

The classical version of this result
is \autocite[Proposition~8.2]{Gallauer18},
which Gallauer proved as a corollary of a more general result there.
This lemma might be seen as a consequence of its variants,
but we here give a different proof using supports.

\begin{proof}
    Suppose that $ I $ is a directed poset
    and that $ \cat{T}^{\otimes} $ is the colimit of a diagram
    $ I \to \Alg_{\E_2}(\Cat{Cat}_{\infty}^{\perf}) $,
    which maps $ i $ to~$ \cat{T}_i^{\otimes} $.
    We wish to show that the morphism $ \injlim_i \Zar(\cat{T}_i)
    \to \Zar(\cat{T}) $ is an equivalence.
    By \autocite[Corollary~3.2.2.5]{HA}
    and the fact that the (nonfull)
    inclusion $ \Cat{Cat}_{\infty}^{\perf} \to \Cat{Cat}_{\infty} $
    preserves filtered colimits,
    $ \cat{T} $ is the colimit of the composite
    $ I \to \Alg_{\E_2}(\Cat{Cat}_{\infty}^{\perf}) \to \Cat{Cat}_{\infty} $.
    Hence it suffices to prove that
    a function from~$ \cat{T} $ to a distributive lattice~$ L $
    is a support
    if the composite
    $ \cat{T}_i \to \cat{T} \to L $ is a support for each $ i $.
    This follows from the definition of a support.
\end{proof}

\subsection{Tensorless tensor triangular geometry}\label{s_tensorless}

In this subsection, we develop
the ``tensorless'' counterpart of the theory
described in the previous section.
This is used in \cref{s_21e9cca7d2}.

First recall that an upper semilattice is
a poset that has finite joins.
A morphism between upper semilattices
is defined to be a function that preserves finite joins.
We let $ \Cat{SLat} $ denote the category
of upper semilattices.

\begin{definition}\label{3c914c1835}
    Suppose that $ \cat{T} $ is an idempotent complete stable $ \infty $-category.
    \begin{enumerate}
        \item
            A \emph{semisupport} for~$ \cat{T} $ is a pair $ (U,s) $ of
            an upper semilattice~$ U $ and a function
            $ s \colon \cat{T} \to U $ satisfying
            conditions \cref{064da024a5,efc38747ae,bd772585de} of \cref{c167a286d1},
            which make sense in this situation.
        \item
            A \emph{thick subcategory} of~$ \cat{T} $
            is an idempotent complete stable full replete subcategory of~$ \cat{T} $.
            It is called \emph{principal} if it is generated,
            as a thick subcategory, by one object.
    \end{enumerate}
\end{definition}

The following is the counterpart of \cref{1745d1eb17}
for semisupports:

\begin{proposition}\label{3232330989}
    For an idempotent complete stable $ \infty $-category $ \cat{T} $,
    the set of
    principal thick subcategories of~$ \cat{T} $
    ordered by inclusion is a semilattice,
    which is (the target of) the initial semisupport.
\end{proposition}

\begin{lemma}\label{36af6900ff}
    For any semisupport $ (U,s) $ and any object $ C \in \cat{T} $,
    the full subcategory $ I \subset \cat{C} $ spanned
    by objects $ D $ satisfying $ s(D) \leq s(C) $
    is a thick subcategory of~$ \cat{T} $.
\end{lemma}

\begin{proof}
    In this proof, we refer to the conditions given in \cref{c167a286d1}.
    From \cref{064da024a5} we have that $ I $ is a full replete subcategory.
    Condition~\cref{efc38747ae} implies $ 0 \in I $
    and \cref{bd772585de} implies that $ I $
    is closed under shifts and (co)fibers. Hence $ I $
    is a stable subcategory.
    Also, from \cref{efc38747ae} we have that $ I $ is idempotent complete,
    which completes the proof.
\end{proof}

\begin{proof}[Proof of \cref{3232330989}]
    For $ C_1, \dotsc, C_n \in \cat{T} $,
    it is easy to see that the join of
    $ \langle C_1 \rangle $, \dots,~$ \langle C_n \rangle $
    can be computed as $ \langle C_1 \oplus \dotsb \oplus C_n \rangle $,
    where $ \langle C \rangle $ denotes the thick subcategory of~$ \cat{T} $
    generated by an object $ C \in \cat{T} $.
    Hence it suffices to show
    that if objects $ C, D \in \cat{T} $ satisfy
    $ s(C) = s(D) $ for some semisupport~$ s $,
    they generate the same thick subcategory.
    This follows from \cref{36af6900ff}.
\end{proof}

\begin{remark}\label{6bae21b07b}
    We can also prove the tensorless counterpart
    of \cref{continuity} by the same argument.
\end{remark}

We state a well-known concrete description of
the free functor $ {\Free} \colon \Cat{SLat} \to \Cat{DLat} $,
which is defined as the left adjoint of the forgetful functor.

\begin{lemma}\label{eb431803ca}
    For an upper semilattice~$ U $,
    we let $ \P(U) $ denote the power set of~$ U $
    ordered by inclusion.
    Then the morphism $ U \to \P(U) $
    that maps $ u $ to $ \{v \in U \mid u \nleq v\} $
    induces a monomorphism of distributive lattices
    $ \Free(U) \hookrightarrow \P(U) $.
\end{lemma}

\section{Tensor triangular geometry of sheaves}\label{s_d6db2a7fad}

In this section, we prove \cref{43db5adc0d},
which is stated in \cref{s_intro}.

\subsection{Tensor triangular geometry of the pointwise monoidal structure}\label{ss_fin-acyc}

First we define a class of categories.
Beware that there are other usages of the word ``acyclic category''
in the literature.

\begin{definition}\label{1080cba1bc}
    An (ordinary) category is called \emph{acyclic}
    if only identity morphisms are
    isomorphisms or endomorphisms in it.
\end{definition}

\begin{example}\label{9e04549152}
    Any poset, considered as a category, is an acyclic category.
\end{example}

Note that every finite acyclic category is $ \omega $-finite
and locally $ \omega $-compact
as an $ \infty $-category,
so that we can apply \cref{1f1e4383ed,5a932f09f2}.

\begin{theorem}\label{79ff4ffbd8}
    Let $ \cat{C}^{\otimes} $ be a big tt-$ \infty $-category and
    $ K $ a finite acyclic category.
    Then we have a canonical isomorphism
    \begin{equation*}
        \Zar(\Fun(K, \cat{C})^{\omega})
        \simeq \Zar(\cat{C}^{\omega})^{K_0},
    \end{equation*}
    where the right hand side denotes
    the power of $ \Zar(\cat{C}^{\omega}) $
    indexed by the set of objects of~$ K $,
    computed in the category $ \Cat{DLat} $.
\end{theorem}

\begin{remark}\label{d4bf78f8df}
    In the language of usual tensor triangular geometry,
    the conclusion of \cref{79ff4ffbd8} just says that
    the Balmer spectrum of $ \Fun(K, \cat{C})^{\omega} $
    is homeomorphic to that of~$ \cat{C}^{\omega} $
    to the power of the cardinality of objects of~$ K $.
\end{remark}

\begin{example}\label{b5665d1379}
    In the case $ \cat{C}^{\otimes} = \Cat{Mod}_k^{\otimes} $
    for some field~$ k $,
    this result is the special case of \autocite[Theorem~2.1.5.1]{LiuSierra13}
    when the quiver is not equipped with relations.
\end{example}

To give the proof of \cref{79ff4ffbd8},
we introduce a notation.

\begin{definition}[used only in this subsection]\label{552e96e6b3}
    In the situation of \cref{79ff4ffbd8},
    suppose that $ k $ is an object of~$ K $.
    We let $ K' $ denote the cosieve generated by~$ k $.
    Let $ X(k) \in \Fun(K, \cat{C}) $ denote
    the left Kan extension of the object of $ \Fun(K', \cat{C}) $
    which is obtained
    as the right Kan extension of the unit
    $ \unit \in \cat{C} \simeq \Fun(\{k\}, \cat{C}) $.
    This object satisfies $ X(k)(k) \simeq \unit $
    and $ X(k)(l) \simeq 0 $ for $ l \neq k $.
\end{definition}

\begin{lemma}\label{037927c088}
    In the situation of \cref{79ff4ffbd8},
    suppose that $ (L, s) $ is
    a support for $ (\Fun(K, \cat{C})^{\omega})^{\otimes} $.
    Then we have $ \bigvee_{k \in K} s(X(k)) = 1 $.
    In other words,
    the object $ \bigoplus_{k \in K} X(k) $
    generates $ \Fun(K, \cat{C})^{\omega} $ as a radical ideal.
\end{lemma}

\begin{proof}
    First, we name objects of~$ K $
    as $ k_1 $, \dots, $ k_n $ so that
    we have $ \Hom_K(k_j,k_i) = \emptyset $ for any $ i < j $.
    This is possible since $ K $ is acyclic.
    For $ 0 \leq i \leq n $, let $ K_i $ denote
    the full subcategory of~$ K $ whose set of objects is $ \{k_j \mid j \leq i\} $.
    We write $ F_i \in \Fun(K, \cat{C})^{\omega} $ for the right Kan extension
    of $ \unit \rvert_{K_i} $, where $ \unit $ denotes
    the unit of $ \Fun(K, \cat{C})^{\otimes} $.

    We wish to prove $ s(F_i) = s(F_{i-1}) \vee s(X(k_i)) $
    for $ 1 \leq i \leq n $,
    which completes the proof since
    $ s(F_n) = s(\unit) = 1 $ and $ s(F_0) = s(0) = 0 $ holds.
    Now since $ K \setminus K_{i-1} $ is a cosieve,
    we get a cofiber sequence
    $ X(k_i) \to F_i \to F_{i-1} $
    by applying \cref{883b124c64}.
    Combining this with
    an equivalence $ X(k_i) \simeq F_i \otimes X(k_i) $,
    we obtain the desired equality.
\end{proof}

\begin{proof}[Proof of \cref{79ff4ffbd8}]
    Let $ \P(K_0) $ denote the power set of
    the set of objects of~$ K $ ordered by inclusion.
    Then there exists a canonical isomorphism
    $ \Zar(\cat{C}^{\omega}) \otimes \P(K_0)
    \simeq \Zar(\cat{C}^{\omega})^{K_0} $ of distributive lattices.

    First, we claim that there exists a (unique) morphism of distributive lattice
    $ \P(K_0) \to \Zar(\Fun(K, \cat{C})^{\omega}) $
    that maps $ \{k\} $ to $ \sqrt{X(k)} $
    for $ k \in K $.
    This follows from the following two observations:
    \begin{enumerate}
        \item
            For $ k \neq l $,
            we have $ X(k) \otimes X(l) \simeq 0 $;
            this can be checked pointwise.
        \item
            The object $ \bigoplus_{k \in K} X(k) $
            generates $ \Fun(K, \cat{C})^{\omega} $ as
            a radical ideal;
            this is the content of \cref{037927c088}.
    \end{enumerate}
    Combining this morphism with the one
    $ \Zar(\cat{C}^{\omega})
    \to \Zar(\Fun(K, \cat{C})^{\omega}) $
    induced by the functor $ K \to {*} $,
    we obtain a morphism $ f \colon \Zar(\cat{C}^{\omega}) \otimes \P(K_0)
    \to \Zar(\Fun(K, \cat{C})^{\omega}) $.

    For each $ k \in K $, the inclusion $ \{k\} \hookrightarrow K $
    induces the morphism
    $ \Zar(\Fun(K, \cat{C})^{\omega}) \to \Zar(\Fun(\{k\}, \cat{C})^{\omega})
    \simeq \Zar(\cat{C}) $.
    From them, we get a morphism
    $ g \colon
    \Zar(\Fun(K, \cat{C})^{\omega})
    \to \Zar(\cat{C}^{\omega})^{K_0}
    \simeq \Zar(\cat{C}^{\omega}) \otimes \P(K_0) $,
    where we use the identification described above.

    To complete the proof, we wish to prove that
    $ g \circ f $ and $ f \circ g $ are identities.
    By construction, it is easy to see that $ g \circ f $ is the identity,
    so it remains to prove the other claim.
    We take an object $ F \in \Fun(K, \cat{C})^{\omega} $
    and for $ k \in K $ let $ F_k $ denote the object of $ \Fun(K, \cat{C})^{\omega} $
    obtained by precomposing the functor
    $ K \to \{k\} \hookrightarrow K $ with~$ F $.
    Unwinding the definitions,
    the assertion $ (g \circ f)\bigl(\sqrt{F}\bigr)
    = \sqrt{F} $ is equivalent to the assertion that
    $ F $ and $ \bigoplus_{k \in K} F_k \otimes X(k) $ generate
    the same radical ideal.
    Since we have an equivalence $ F \otimes X(k) \simeq F_k \otimes X(k) $
    for each $ k \in K $, this follows from \cref{037927c088}.
\end{proof}

\subsection{Main result}\label{ss_shv-coh}

In order for \cref{43db5adc0d} to make sense,
we need to clarify what
the pointwise $ \E_2 $-monoidal structure on
$ \Shv_{\cat{C}}(X) $ is.

\begin{proposition}\label{9795957f03}
    For a coherent topological space~$ X $,
    the $ \infty $-topos $ \Shv(X) $ is compactly generated.
    Moreover, finite products of compact objects are again compact.
\end{proposition}

\begin{proof}
    The first assertion is the content of \autocite[Proposition~6.5.4.4]{HTT}.
    Let $ L $ denote the distributive lattice of quasicompact open subsets of~$ X $.
    We have $ \Shv(L) \simeq \Shv(X) $,
    where $ L $ is equipped with the induced Grothendieck
    topology (see \cref{b549dfa54c}).
    We let $ L' \colon \PShv(L) \to \Shv(L) $ denote
    the sheafification functor.
    It follows from the proof of \autocite[Proposition~6.5.4.4]{HTT}
    that $ \Shv(L)^{\omega} $ is the smallest full subcategory
    that contains the image of $ \PShv(L)^{\omega} $ under~$ L' $
    and is closed under finite colimits and retracts.
    Since finite products preserve (finite) colimits in each variable in $ \Shv(L) $
    and the functor~$ L' $ preserves finite products,
    it suffices to show that finite products of compact objects
    are again compact in $ \PShv(L) $.
    This follows from \cref{3752cd1622} since $ L $ has finite products.
\end{proof}

Using the equivalence $ \Shv_{\cat{C}}(X) \simeq \Shv(X) \otimes \cat{C} $
and \cref{8e9b197f4f}, we can equip
an $ \E_2 $-monoidal structure on $ \Shv_{\cat{C}}(X) $.

\begin{corollary}\label{03eaee5199}
    For a big tt-$ \infty $-category~$ \cat{C} $
    and a coherent topological space~$ X $,
    the $ \E_2 $-monoidal $ \infty $-category $ \Shv_{\cat{C}}(X)^{\otimes} $,
    defined as above,
    is a big tt-$ \infty $-category.
\end{corollary}

To state the main theorem,
we recall
basic facts on Boolean algebras.
See \autocite[Section~II.4]{Johnstone82} for details.
First recall
that Boolean algebras form a reflective subcategory of $ \Cat{DLat} $,
which we denote by $ \Cat{BAlg} $.
The left adjoint of the inclusion $ \Cat{BAlg} \hookrightarrow \Cat{DLat} $
is called the Booleanization functor,
which we denote by~$ \B $.
For a coherent topological space~$ X $,
the spectrum of its Booleanization of the distributive lattice of
quasicompact open subsets of~$ X $ is
the Stone space whose topology is the constructible topology
(also referred to as the patch topology) of~$ X $.
Hence using \cref{5f0590fbb0}, \cref{43db5adc0d} can be rephrased as follows:

\begin{theorem}\label{7fd46cb929}
    Let $ \cat{C}^{\otimes} $ be a big tt-$ \infty $-category and
    $ L $ a distributive lattice.
    Then we have a canonical isomorphism
    \begin{equation*}
        \Zar(\Shv_{\cat{C}}(\Spec(L))^{\omega})
        \simeq \Zar(\cat{C}^{\omega}) \otimes \B(L).
    \end{equation*}
\end{theorem}

The proof uses the following notion:

\begin{definition}\label{95b1e699de}
    For a poset $ P $,
    the \emph{Alexandroff topology}
    is a topology on the underlying set of~$ P $
    whose open sets are cosieves (or equivalently, upward closed subsets).
    Let $ \Alex(P) $ denote the set~$ P $
    equipped with this topology.
\end{definition}

\begin{lemma}\label{19837320e6}
    For a big tt-$ \infty $-category~$ \cat{C}^{\otimes} $
    and a finite poset~$ P $,
    there exists a canonical equivalence between
    big tt-$ \infty $-categories
    \begin{equation*}
        \Fun(P, \cat{C})^{\otimes}
        \simeq \Shv_{\cat{C}}(\Alex(P))^{\otimes}.
    \end{equation*}
\end{lemma}

The proof relies on a result obtained in \cref{s_basis}.

\begin{proof}
    We have the desired equivalence by
    taking
    the tensor product of (the cartesian $ \E_2 $-monoidal refinement of)
    the equivalence $ \PShv(P^{\op}) \to \Shv(\Alex(P)) $
    obtained in \cref{ec5a5ba37f} with~$ \cat{C}^{\otimes} $.
\end{proof}

\begin{lemma}\label{2297e48463}
    For a big tt-$ \infty $-category~$ \cat{C}^{\otimes} $,
    the functor
    from~$ \Cat{DLat} $ to~$ \Alg_{\E_2}(\Cat{Pr}_{\omega}^{\st}) $
    that maps $ L $ to $ (\Shv_{\cat{C}}(\Spec(L)))^{\otimes} $
    preserves filtered colimits.
\end{lemma}

\begin{proof}
    We first note that
    the composite of forgetful functors
    $ \Alg_{\E_2}(\Cat{Pr}_{\omega}^{\st})
    \to \Cat{Pr}_{\omega}^{\st} \to \Cat{Pr}_{\omega} $
    preserves sifted colimits and conservative.
    Since limits in the large $ \infty $-categories
    $ \Cat{Pr}^{\op} $ and $ \Cat{Pr}_{\omega}^{\op} $ are both computed in
    the very large $ \infty $-category of large $ \infty $-categories,
    the forgetful functor $ \Cat{Pr}_{\omega} \to \Cat{Pr} $
    preserves colimits and obviously conservative.
    Hence we are reduced to showing that the following composite
    preserves filtered colimits:
    \begin{equation*}
        \Cat{DLat}
        \xrightarrow{\Idl} \Cat{Loc}^{\op}
        \xrightarrow{{\Shv}^{\op}} \Cat{Top}_{\infty}^{\op}
        \xrightarrow{\textnormal{forgetful}} \Cat{Pr}
        \xrightarrow{\X \otimes \cat{C}} \Cat{Pr},
    \end{equation*}
    where $ \Cat{Top}_{\infty} $ denote
    the very large $ \infty $-category of $ \infty $-toposes
    whose morphisms are geometric morphisms.
    Now we can check that each functor
    preserves filtered colimits as follows:
    \begin{enumerate}
        \item
            The first functor preserves colimits by \cref{6d78f23da1}.
        \item
            The second functor preserves colimits since
            $ 0 $-localic $ \infty $-toposes
            form a reflective subcategory of $ \Cat{Top}_{\infty} $.
        \item
            The third functor preserves filtered colimits
            since cofiltered limits in $ \Cat{Top}_{\infty}^{\op} $ can be
            computed in
            the very large $ \infty $-category of large $ \infty $-categories.
        \item
            The fourth functor preserves colimits by~\cref{dee93321da}
            of \cref{tensor}.
    \end{enumerate}
    Hence the composite also preserves filtered colimits.
\end{proof}

\begin{proof}[Proof of \cref{7fd46cb929}]
    Since finitely generated distributive lattice has
    a finite number of objects,
    we can write $ L $ as a filtered colimit
    of finite distributive lattices.
    Hence
    by \cref{2297e48463,continuity},
    it suffices to consider the case when $ L $ is finite.
    Let $ P $ be a poset of points of $ \Spec(L) $
    with the specialization order;
    that is, the partial order in which $ p \leq q $
    if and only if the point~$ p $
    is contained in the closure of the singleton~$ \{q\} $.
    Since $ \Spec(L) $ is finite,
    there is a canonical homeomorphism $ \Spec(L) \simeq \Alex(P) $.
    Booleanizing their associated distributive lattices,
    we have that~$ \B(L) $ is canonically isomorphic to
    the power set of~$ P $
    ordered by inclusion.
    Then applying \cref{19837320e6},
    we obtain the desired equivalence as a corollary of \cref{79ff4ffbd8}.
\end{proof}

\section{Tensor triangular geometry of the Day convolution}\label{s_21e9cca7d2}

In this section, we prove \cref{5e265d653c},
which is a generalization of \cref{cbaf04a35f}.

\subsection{Partially ordered abelian groups}\label{ss_66b4e11377}

We begin with reviewing some notions used in the theory
of partially ordered abelian groups.

\begin{definition}\label{464cc710a7}
    A \emph{partially ordered abelian group}
    is an abelian group object of the category of posets;
    that is, an abelian group~$ A $ equipped with
    a partial order in which the map $ a + \X $ is order preserving
    for any $ a \in A $.
\end{definition}

We can regard a partially ordered abelian group
as a symmetric monoidal poset.

\begin{definition}\label{abac4be71f}
    Let $ A $ be a partially ordered abelian group.
    \begin{enumerate}
        \item
            The submonoid
            $ A_{\geq 0} = \{a \in A \mid a \geq 0\} $ is called
            the \emph{positive cone} of~$ A $.
        \item
            The subgroup $ A^{\circ} = \{a - b \mid a, b \in A_{\geq 0}\} $
            is called the \emph{identity component} of~$ A $.
            As the name suggests, this is
            the connected component containing $ 0 $
            when $ A $ is regarded as a category by its partial order.
        \item
            If $ A^{\circ} = A $ holds and $ A_{\geq 0} $ has finite joins,
            $ A $ is called \emph{lattice ordered}.
            This is equivalent to the condition that $ A $ has
            binary joins (but beware that $ A $ does not have
            the nullary join unless $ A $ is trivial).
            Note that in this case $ A $ also has binary meets,
            which are
            computed as $ a \wedge b = a + b - (a \vee b) $ for $ a, b \in A $.
    \end{enumerate}
\end{definition}

\begin{example}\label{0a6fe8036e}
    For a cardinal~$ \kappa $,
    the (categorical) product ordering on~$ \ZZ^{\kappa} $
    defines a lattice ordered abelian group.

    The assignment $ f \colon (x_1,x_2) \mapsto (x_1,x_2,x_1+x_2) $
    gives a morphism $ \ZZ^2 \to \ZZ^3 $
    of partially ordered abelian groups,
    but does not define a morphism of unbounded lattice:
    Indeed, we have
    $ f((1,0)) \vee f((0,1)) = (1,1,1) \neq (1,1,2) = f((1,0) \vee (0,1)) $.
\end{example}

\begin{example}\label{nonlo}
    There are many partially ordered abelian groups
    that are connected and not lattice ordered.
    We here give two relatively simple examples.
    The first is $ \ZZ $ with the ordering that
    makes its positive cone
    to be the submonoid generated by~$ 2 $ and~$ 3 $.
    The second is $ \ZZ \times \ZZ/2 $ with the ordering that
    makes its positive cone to be the submonoid
    generated by~$ (1, 0) $ and~$ (1, 1) $.
\end{example}

\subsection{The Archimedes semilattice}\label{ss_5edc93849b}

In this subsection,
we introduce the notion of the Archimedes semilattice
of a partially ordered abelian group.

\begin{definition}\label{60c5499fb9}
    For a partially ordered abelian group $ A $,
    a submonoid $ J \subset A_{\geq 0} $ is called an \emph{ideal} of~$ A_{\geq 0} $
    if it is downward closed;
    that is, if $ a, b \in A_{\geq 0} $ satisfies $ a \leq b $ and $ b \in J $,
    we have $ a \in J $.
    An ideal $ J $ is called \emph{principal} if it is generated
    as an ideal by an element of~$ A_{\geq 0} $.
\end{definition}

\begin{proposition}\label{c5f65cc35d}
    For a partially ordered abelian group~$ A $,
    the set of principal ideals of~$ A_{\geq 0} $ ordered by inclusion
    is an upper semilattice.
\end{proposition}

\begin{proof}
    It is easy to see that
    $ \langle a_1 + \dotsb + a_n \rangle $
    is a join of $ \langle a_1 \rangle $, \dots, $ \langle a_n \rangle $
    for $ a_1, \dotsc, a_n \in A_{\geq 0} $,
    where $ \langle a \rangle $ denotes
    the ideal of~$ A_{\geq 0} $ generated by an element $ a \in A_{\geq 0} $.
\end{proof}

\begin{definition}\label{8560b04872}
    For a partially ordered abelian group $ A $,
    the upper semilattice of principal ideals of~$ A_{\geq 0} $
    is denoted by $ \Arch(A) $.
    We call this the \emph{Archimedes semilattice} of~$ A $.
    Note that this only depends on its positive cone $ A_{\geq 0} $,
    regarded as a partially ordered abelian monoid.
\end{definition}

\begin{remark}\label{77fd493ebd}
    There is a characterization of the Archimedes semilattice
    similar to that of the Zariski lattice given in \cref{1745d1eb17}:
    Namely, the Archimedes semilattice $ \Arch(A) $ is initial among pairs $ (U, s) $
    where $ U $ is an upper semilattice
    and $ s \colon A_{\geq 0} \to U $ is a function
    satisfying the following conditions:
    \begin{enumerate}
        \item
            For $ a_1, \dotsc, a_n \in A_{\geq 0} $, we have
            $ s(a_1 + \dotsb + a_n) = s(a_1) \vee \dotsb \vee s(a_n) $.
        \item
            If $ a, b \in A_{\geq 0} $ satisfies $ a \leq b $,
            we have $ s(a) \leq s(b) $.
    \end{enumerate}
\end{remark}

\begin{example}\label{9ad2e6fd5}
    If $ A $ is a totally ordered abelian group,
    its Archimedes semilattice
    $ \Arch(A) $ consists of all Archimedean classes of~$ A $
    and the singleton~$ \{0\} $.
    This observation justifies the name.
    In particular,
    if $ A $ is nonzero Archimedean,
    we have $ \Arch(A) \simeq \{0 < 1\} $.
\end{example}

\begin{example}\label{d50762581e}
    Any Riesz space~$ R $ can be regarded
    as a lattice ordered abelian group
    in a trivial way.
    There is a bijective (and order preserving)
    correspondence between
    (principal) ideals of~$ R_{\geq 0} $ and
    those of~$ R $ in the usual sense.
\end{example}

\subsection{Tensorless tensor triangular geometry with actions}\label{ss_ss}

In this subsection, we construct the comparison map,
which we prove to be an isomorphism under
some mild assumptions.

\begin{proposition}\label{8388023eda}
    Let $ \cat{C}^{\otimes} $ be a big tt-$ \infty $-category
    and $ A $ a partially ordered abelian group.
    Then we have a canonical morphism of distributive lattices
    \begin{equation*}
        f \colon
        \Zar(\cat{C}^{\omega}) \otimes \Free(\Arch(A))
        \to \Zar(\Fun(A, \cat{C})^{\omega}),
    \end{equation*}
    where $ \Free \colon \Cat{SLat} \to \Cat{DLat} $
    denotes the left adjoint to the forgetful functor.
\end{proposition}

It is convenient to
regard the $ \infty $-category $ \Fun(A, \cat{C}) $ as equipped
with the action of~$ A $, which is described in the following definition:

\begin{definition}\label{e9bc641989}
    Suppose that $ \cat{C} $
    is a compactly generated stable $ \infty $-category
    and $ A $ is a partially ordered abelian group.

    Then precomposition
    with the map $ (a,b) \mapsto b-a $ induces a functor
    $ \Fun(A, \cat{C}) \to \Fun(A^{\op} \times A, \cat{C}) $,
    which can be seen as a functor from $ \Fun(A, \cat{C}) \times A^{\op} $
    to $ \Fun(A, \cat{C}) $.
    We write $ F\{a\} $ for the value of this functor at $ (F, a) $
    and $ F\{a/b\} $ for the cofiber of the map $ F\{b\} \to F\{a\} $,
    which is only defined when $ a \leq b $;
    concretely, $ F\{a\} $ is an object satisfying $ F\{a\}(b) \simeq F(b-a) $
    for $ b \in A $.

    We call a semisupport~$ s $ for $ \Fun(A, \cat{C})^{\omega} $
    an \emph{$A$-semisupport} if
    for $ a \in A $ and $ F \in \Fun(A, \cat{C})^{\omega} $,
    we have $ s(F\{a\}) = s(F) $.
    Similarly, we call a thick subcategory of $ \Fun(A, \cat{C})^{\omega} $
    a \emph{thick $A$-subcategory} if it is stable
    under the operation $ F \mapsto F\{a\} $ for any $ a \in A $.
\end{definition}

From now on, we abuse notation by identifying the object~$ C \in \cat{C}^{\omega} $
with its left Kan extension
along the inclusion $ 0 \hookrightarrow A $ of partially
ordered abelian groups.
Especially, we do not distinguish the units of~$ \cat{C}^{\otimes} $
and $ \Fun(A, \cat{C})^{\otimes} $, which we denote by~$ \unit $.

\begin{example}\label{12dc5c4509}
    Suppose that $ \cat{C}^{\otimes} $ is a big tt-$ \infty $-category.
    For any object $ F \in \Fun(A, \cat{C})^{\omega} $
    and any element $ a \in A $ we have
    $ F\{a\} \simeq F \otimes \unit\{a\} $
    and that $ \unit\{a\} $ is invertible with inverse $ \unit\{-a\} $.
    Thus any support for $ (\Fun(A, \cat{C})^{\omega})^{\otimes} $
    can be regarded as an $ A $-semisupport.
    More generally, for any object $ G \in \Fun(A, \cat{C})^{\omega} $
    and any support~$ s $, the assignment $ F \mapsto s(G \otimes F) $ defines
    an $ A $-semisupport for $ \Fun(A, \cat{C})^{\omega} $.
\end{example}

\begin{lemma}\label{ss-main}
    In the situation of \cref{e9bc641989},
    suppose that $ s $ is
    an $ A $-semisupport for $ \Fun(A, \cat{C})^{\omega} $
    and $ F $ is an object of $ \Fun(A, \cat{C})^{\omega} $.
    Then the following assertions hold:
    \begin{enumerate}
        \item
            \label{ss-additive}
            For $ a_1, \dotsc, a_n \in A_{\geq 0} $, we have
            $ s(F\{0/(a_1 + \dotsb + a_n)\})
            = s(F\{0/a_1\}) \vee \dotsb \vee s(F\{0/a_n\}) $.
        \item
            \label{ss-monotone}
            For $ a, b \in A_{\geq 0} $ satisfying
            $ a \leq b $, we have
            $ s(F\{0/a\}) \leq s(F\{0/b\}) $.
    \end{enumerate}
\end{lemma}

\begin{proof}
    We first prove \cref{ss-monotone}.
    Consider the following diagram,
    in which all the rows and columns are cofiber sequences:
    \begin{equation*}
        \xymatrix{
            F\{a+b\} \ar[r] \ar[d]
            & F\{a\} \ar[r] \ar[d]
            & F\{a/(a+b)\} \ar[d]
            \\
            F\{b\} \ar[r] \ar[d]
            & F\{0\} \ar[r] \ar[d]
            & F\{0/b\} \ar[d]
            \\
            F\{b/(a+b)\}
            \ar[r]^{f}
            & F\{0/a\}
            \ar[r]
            & C\rlap{.}
        }
    \end{equation*}
    Since there exists an equivalence
    $ F\{a/(a+b)\} \simeq F\{0/b\}\{a\} $,
    we have $ s(F\{a/(a+b)\}) = s(F\{0/b\}) $.
    Similarly we have $ s(F\{b/(a+b)\}) = s(F\{0/a\}) $.
    Using the right cofiber sequence, we have
    $ s(C) \leq s(F\{0/b\}) $.
    We complete the proof by showing $ s(C) = s(F\{0/a\}) $.
    To prove this, it suffices to show that the morphism~$ f $
    in the diagram is zero.
    This follows from the fact that
    the left top square can be decomposed as follows:
    \begin{equation*}
        \xymatrix{
            F\{a+b\} \ar[r] \ar[d]
            & F\{a\} \ar@{=}[r] \ar@{=}[d]
            & F\{a\} \ar[d]
            \\
            F\{b\} \ar[r]
            & F\{a\} \ar[r]
            & F\{0\}\rlap{.}
        }
    \end{equation*}

    We then prove \cref{ss-additive}.
    The case $ n = 0 $ is trivial.
    Hence it suffices to consider the case $ n = 2 $.
    Consider the following diagram:
    \begin{equation*}
        \xymatrix{
            F\{a_1+a_2\} \ar@{=}[r] \ar[d]
            & F\{a_1+a_2\} \ar[r] \ar[d]
            & 0 \ar[d]
            \\
            F\{a_2\} \ar[r] \ar[d]
            & F\{0\} \ar[r] \ar[d]
            & F\{0/a_2\} \ar@{=}[d]
            \\
            F\{a_2/(a_1+a_2)\}
            \ar[r]
            & F\{0/(a_1+a_2)\}
            \ar[r]
            & F\{0/a_2\}\rlap{.}
        }
    \end{equation*}
    Since all the other rows and columns are cofiber sequences,
    so is the bottom row.
    Hence
    we have
    $ s(F\{0/(a_1+a_2)\}) \leq s(F\{a_2/(a_1+a_2)\}) \vee s(F\{0/a_2\})
    = s(F\{0/a_1\}) \vee s(F\{0/a_2\}) $.
    On the other hand, applying \cref{ss-monotone},
    we have $ s(F\{0/a_1\}) \vee s(F\{0/a_2\})
    \leq s(F\{0/(a_1+a_2)\}) $.
    Therefore, the desired equality follows.
\end{proof}

\begin{proof}[Proof of \cref{8388023eda}]
    The left Kan extension along the inclusion $ 0 \hookrightarrow A $
    defines a morphism
    $ \Zar(\cat{C}^{\omega}) \to \Zar(\Fun(A, \cat{C})^{\omega}) $.
    By \cref{ss-main},
    the assignment $ a \mapsto \sqrt{\unit \{0/a\}} $ satisfies the conditions
    given in \cref{77fd493ebd}, so we have
    a morphism $ \Free(\Arch(A)) \to \Zar(\Fun(A, \cat{C})^{\omega}) $.
    Combining these two, we obtain the desired morphism.
\end{proof}

Now we study $ A $-semisupports in more detail.
First, by mimicking the proof of \cref{3232330989},
we can deduce the following:

\begin{proposition}\label{b78d433be4}
    In the situation of \cref{e9bc641989},
    principal thick $ A $-subcategories form
    an upper semilattice by inclusion
    and the assignment that
    takes an object of $ \Fun(A, \cat{C})^{\omega} $ to
    the thick $ A $-subcategory generated by it
    defines an $ A $-semisupport.
    Furthermore, it is an initial $ A $-semisupport.
\end{proposition}

If $ A $ is lattice ordered,
the initial $ A $-semisupport has simple generators.

\begin{proposition}\label{ss-gen}
    In the situation of \cref{e9bc641989},
    we furthermore assume that $ A $ is lattice ordered.
    Then the target of the initial $ A $-semisupport
    described in \cref{b78d433be4}
    is generated as an upper semilattice
    by thick $ A $-subcategories
    generated by an object of the form $ C\{0/a_1\} \dotsb \{0/a_n\} $
    with $ C \in \cat{C}^{\omega} $
    and $ a_1, \dotsc, a_n \in A_{\geq 0} $
    satisfying
    $ a_i \wedge a_j = 0 $ if $ i \neq j $.
\end{proposition}

We prove this in \cref{ss_proof}.
Note that the conclusion also holds
if we only require the positive cone to have binary joins;
see the proof of \cref{to-comp}.

\subsection{Main theorem}\label{ss_1359ae9a2e}

We now state the main result of this section.

\begin{theorem}\label{5e265d653c}
    In the situation of \cref{8388023eda},
    suppose furthermore that $ A_{\geq 0} $ has finite joins.
    Then the morphism~$ f $ is an isomorphism.
\end{theorem}

\begin{question}\label{3824a9f782}
    In \cref{5e265d653c},
    what happens if $ A $ does not satisfy the hypothesis?
    By \cref{to-comp}, we may assume that $ A $ is connected.
    To consider the case $ A $ is one of
    the examples given in \cref{nonlo}
    would be a starting point.
\end{question}

\begin{example}\label{8e997e2f3a}
    Claim \cref{04e366edb0} of \cref{cbaf04a35f}
    is a direct consequence of \cref{5e265d653c}.

    Since $ \Free(\{0<1\}) $
    is the linearly ordered set
    consisting of three elements,
    $ \Spec(\Free(\{0<1\})) $
    is homeomorphic to the Sierpi\'nski space.
    Hence by using \cref{5f0590fbb0,9ad2e6fd5},
    we can deduce \cref{e6bd26145c} of \cref{cbaf04a35f}.
\end{example}

Since the inclusion $ A^{\circ} \hookrightarrow A $
induces an isomorphism $ \Arch(A^{\circ}) \simeq \Arch(A) $,
\cref{5e265d653c} is a consequence of the following two results:

\begin{proposition}\label{to-comp}
    Let $ \cat{C}^{\otimes} $ be a big tt-$ \infty $-category
    and $ A $ a partially ordered abelian group.
    Then the morphism $ i \colon \Zar(\Fun(A^{\circ}, \cat{C})^{\omega})
    \to \Zar(\Fun(A, \cat{C})^{\omega}) $ induced by
    the inclusion $ A^{\circ} \hookrightarrow A $
    is an equivalence.
\end{proposition}

\begin{proposition}\label{lo-case}
    In the situation of \cref{8388023eda},
    suppose furthermore that $ A $ is lattice ordered.
    Then the morphism~$ f $ is an isomorphism.
\end{proposition}

We conclude this subsection by showing \cref{to-comp};
we prove \cref{lo-case} in the next subsection.

\begin{proof}[Proof of \cref{to-comp}]
    In this proof,
    we regard $ \Fun(A^{\circ}, \cat{C})^{\omega} $
    as a full subcategory of $ \Fun(A, \cat{C})^{\omega} $
    by left Kan extension.

    Let $ J $ denote the quotient $ A/A^{\circ} $.
    For each $ j \in J $ we choose an element $ a_j \in j $.
    Then as a category,
    we can identify $ A $ with
    $ \coprod_{j \in J} (A^{\circ} + a_j) $.
    Hence combining the translations
    $ A + a_j \to A $ for all $ j \in J $,
    we have a functor $ A \to A^{\circ} $.
    This defines a functor $ \Fun(A, \cat{C})^{\omega}
    \to \Fun(A^{\circ}, \cat{C})^{\omega} $ by left Kan extension;
    concretely, it is given by the formula
    $ \bigoplus_{j \in J} F_j \{a_j\} \mapsto \bigoplus_{j \in J} F_j $
    for any family $ (F_j)_{j \in J} $ of objects
    of $ \Fun(A^{\circ}, \cat{C})^{\omega} $ such that
    $ F_j $ is zero
    except a finite number of indices $ j \in J $.

    Let $ s $ denote the composite
    $ \Fun(A, \cat{C})^{\omega} \to \Fun(A^{\circ}, \cat{C})^{\omega}
    \to \Zar(\Fun(A^{\circ}, \cat{C})^{\omega}) $,
    where the second map is the initial support.
    We wish to show that $ s $ is a support for
    $ (\Fun(A, \cat{C})^{\omega})^{\otimes} $
    by checking that the conditions given in \cref{c167a286d1}
    are satisfied.
    The only nontrivial point is to prove that
    $ s $ satisfies condition~\cref{dbba1d4dc6}.
    Since the case $ n = 0 $ follows
    from the fact that $ \unit\{-a_{[0]}\} $ is invertible,
    it is sufficient to consider the case $ n = 2 $.
    We take two objects $ F, G \in \Fun(A, \cat{C})^{\omega} $
    and decompose them
    as $ F \simeq \bigoplus_{j \in J} F_j\{a_j\} $
    and $ G \simeq \bigoplus_{j \in J} G_j\{a_j\} $
    using two families of objects $ (F_j)_{j \in J} $, $ (G_j)_{j \in J} $
    of $ \Fun(A^{\circ}, \cat{C})^{\omega} $ such that
    both $ F_j $ and $ G_j $ are zero
    except a finite number of indices $ j \in J $.
    To prove $ s(F \otimes G) = s(F) \wedge s(G) $,
    unwinding the definitions,
    we need to show that
    the equality
    \begin{equation*}
        \bigvee_{j, j' \in J}
        \sqrt{(F_j \otimes G_{j'})\{a_j+a_{j'}-a_{j+j'}\}}
        =
        \bigvee_{j \in J} \sqrt{F_j}
        \wedge
        \bigvee_{j \in J} \sqrt{G_j}
    \end{equation*}
    holds in $ \Zar(\Fun(A^{\circ}, \cat{C})^{\omega}) $.
    This follows from the observation made in
    \cref{12dc5c4509}.

    Hence we obtain a morphism
    $ r \colon \Zar(\Fun(A, \cat{C})^{\omega})
    \to \Zar(\Fun(A^{\circ}, \cat{C})^{\omega}) $
    from~$ s $.
    The composite $ r \circ i $ is the identity
    since we have
    $ (r \circ i)\bigl(\sqrt{F}\bigr) = \sqrt{F\{-a_{[0]}\}} = \sqrt{F} $
    for $ F \in \Fun(A^{\circ}, \cat{C})^{\omega} $.
    To prove that
    $ i \circ r $ is the identity,
    unwinding the definitions,
    it suffices to show that
    $ \sqrt{\bigoplus_{j \in J} F_j\{a_j\}} $ equals to
    $ \sqrt{\bigoplus_{j \in J} F_j} $
    for any family $ (F_j)_{j \in J} $ of objects
    of $ \Fun(A^{\circ}, \cat{C})^{\omega} $ such that
    $ F_j $ is zero
    except a finite number of indices $ j \in J $.
    This also follows from the observation made in \cref{12dc5c4509}.
\end{proof}

\subsection{Postponed proofs}\label{ss_proof}

In this subsection, we give the proofs of
\cref{lo-case,ss-gen}.
First we introduce some terminology.

\begin{definition}[used only in this subsection]\label{8e21b697ce}
    Suppose that $ A $ is a lattice ordered abelian group.
    \begin{enumerate}
        \item
            \label{bb673d9b4d}
            We call a subset $ B \subset A $ \emph{saturated} if it is finite
            and closed under binary joins.
            Note that for every finite set $ B \subset A $
            we can find the smallest saturated subset of~$ A $
            containing~$ B $.
        \item
            \label{6039d772a4}
            Let $ \cat{C} $ be a compactly generated stable $ \infty $-category.
            By applying \cref{883b124c64}
            to the cosieve of~$ A $ generated by a single element~$ a \in A $,
            we obtain a presentation
            of $ \Fun(A, \cat{C}) $ as a recollement.
            Hence for $ F \in \Fun(A, \cat{C}) $
            we have a cofiber sequence $ j_!j^*F \to F \to i_*i^*F $,
            where we use the notation of \cref{recol-spl}.
            We write $ F_{\sgeq a} $ and~$ F_{\sngeq a} $
            for~$ j_!j^*F $ and~$ i_*i^*F $, respectively.
    \end{enumerate}
\end{definition}

The proof of \cref{ss-gen} relies on the following lemma:

\begin{lemma}\label{ssspl-1}
    Suppose that $ A $ is a lattice ordered abelian group,
    $ \cat{C} $ is a compactly generated stable $ \infty $-category,
    and $ s $ is an $ A $-semisupport for $ \Fun(A, \cat{C})^{\omega} $.
    Then
    for $ a \in A $
    and $ F \in \Fun(A, \cat{C})^{\omega} $, we have
    $ s(F) = s(F_{\sgeq a}) \vee s(F_{\sngeq a}) $
\end{lemma}

We first prove the following special case:

\begin{lemma}\label{ssspl-0}
    In the situation of \cref{ssspl-1},
    suppose that $ a' \leq a $
    are elements of~$ A $ and
    $ B \subset A $ is a saturated subset
    satisfying $ b \wedge a \in \{a', a\} $ for $ b \in B $.
    If an object $ F \in \Fun(A, \cat{C})^{\omega} $
    is obtained as the left Kan extension of~$ F \rvert_B $,
    we have $ s(F) = s(F_{\sgeq a}) \vee s(F_{\sngeq a}) $.
\end{lemma}

\begin{proof}
    We may assume that $ a' = 0 $
    and $ B $ contains $ 0 $,
    which automatically becomes the least element of~$ B $.

    We again consider the recollement description
    of $ \Fun(A, \cat{C}) $
    given in~\cref{6039d772a4} of \cref{8e21b697ce}
    and continue to use the notation of \cref{recol-spl}.

    We first show that
    $ j^*((i_!i^*F)\{0/a\}) $ is zero,
    where $ i_! $ denote the left adjoint of~$ i^* $.
    This is equivalent to the assertion that for any $ c \in A_{\geq 0} $
    the morphism $ (i_!i^*F)(c) \to (i_!i^*F)(c+a) $ is an equivalence.
    We may assume that $ b \wedge a = 0 $ holds for any $ b \in B $
    to prove this.
    Unwinding the definitions, it is enough to show that
    $ \{b \in B \mid b \leq c\} $
    and $ \{b \in B \mid b \leq c+a\} $ have the same greatest element.
    Let $ b $ be the greatest element of the latter set.
    Then we have
    $ b = b \wedge (c+a) = c + ((b-c) \wedge a) \leq c + (b \wedge a)= c $,
    which means that $ b $ belongs to the former set.

    Then we consider the following diagram:
    \begin{equation*}
        \xymatrix{
            i^*((i_!i^*F)\{0/a\})
            \ar[r]_{\simeq} \ar[d]
            & i^*(F\{0/a\})
            \ar[d]
            \\
            i^*j_*j^*((i_!i^*F)\{0/a\})
            \ar[r]
            & i^*j_*j^*(F\{0/a\})\rlap{.}
        }
    \end{equation*}
    By what we have shown above, the bottom left object is zero,
    so that the right vertical morphism is zero.
    By applying \cref{recol-spl}, we have that
    $ F_{\sngeq a} \simeq (F\{0/a\})_{\sngeq a} $ is a direct sum of $ F\{0/a\} $.
    Hence we have $ s(F) \geq s(F\{0/a\}) \geq s(F_{\sngeq a}) $,
    which completes the proof.
\end{proof}

\begin{proof}[Proof of \cref{ssspl-1}]
    We take a saturated subset $ B \subset A $
    such that $ F $ is equivalent to the left Kan extension
    of~$ F \rvert_{B} $.
    We may assume that
    $ B $ contains $ 0 $ as the least element.
    By replacing $ a $ with~$ a \vee 0 $,
    we may assume that $ a \geq 0 $
    and also $ a \in B $.
    Now we take a maximal chain
    $ 0 = b_0 < \dotsb < b_n = a $
    in~$ B $.
    Then
    we can apply \cref{ssspl-0} iteratively
    to obtain an inequality
    \begin{equation*}
        s(F)
        = s(F_{\sgeq 0})
        = s(F_{\sgeq b_0})
        \geq \dotsb
        \geq s(F_{\sgeq b_n})
        = s(F_{\sgeq a}),
    \end{equation*}
    which completes the proof.
\end{proof}

\begin{proof}[Proof of \cref{ss-gen}]
    First we take a finite subset $ B \subset A $
    such that $ B $ is closed under
    binary joins and meets
    and $ F $ is equivalent to the left Kan extension of~$ F \rvert_B $.
    For $ b \in B $,
    we can take distinct elements $ a_1, \dotsc, a_n \in A_{\geq 0} $
    (possibly $ n = 0 $)
    such that $ \{b + a_1, \dotsc, b + a_n\} $
    is the set of minimal elements of $ \{c \in B \mid b < c \} $.
    By the assumption on~$ B $, we have that $ a_i \wedge a_j = 0 $
    for $ i \neq j $.
    Let $ F_b $ denote the object
    $ (\dotsb(F_{\sgeq b})_{\sngeq b + a_1}\dotsb)_{\sngeq b + a_n} $.
    Then applying \cref{ssspl-1} iteratively,
    we have $ s(F) = \bigvee_{b \in B} s(F_b) $.
    Thus we wish to show
    that $ F_b\{-b\} $ is equivalent to
    $ F(b) \{0/a_1\} \dotsb \{0/a_n\} $
    for $ b \in B $
    to complete the proof.

    We may assume that $ b = 0 $ and $ F = F_{\sgeq 0} $.
    We define a subset of~$ A $ by
    $ B' = \bigl\{\sum_{i \in I} a_i \bigm| I \subset \{1, \dotsc, n\}\bigr\} $;
    note that here $ \sum_{i \in I} a_i $ equals to the join $ \bigvee_{i \in I} a_i $
    taken in $ A_{\geq 0} $.
    By induction
    we can see that $ F(0)\{0/a_1\} \dotsb \{0/a_n\} $
    is equivalent to
    the left Kan extension of the object
    of~$ \Fun(B', \cat{C}) $
    which is the right Kan extension
    of $ F(0) \in \Fun(\{0\}, \cat{C}) $.
    Using the equivalence
    \begin{equation*}
        F_0
        \simeq
        (\dotsb(F(0)\{0/a_1\}\dotsb\{0/a_n\})_{\sngeq a_1}\dotsb)_{\sngeq a_n},
    \end{equation*}
    we are reduced to showing
    that $ (F(0)\{0/a_1\}\dotsb\{0/a_n\})(c) \simeq 0 $
    if $ c \in A $ satisfies $ c \geq a_i $ for some $ i $.
    This follows from the above description.
\end{proof}

Finally, we give the proof of \cref{lo-case}.

\begin{proof}[Proof of \cref{lo-case}]
    For a saturated subset $ B \subset A $
    and $ a \in B $, we make the following definition:
    \begin{equation*}
        \Theta(B,a) = \{J \in \Arch(A) \mid (a + J) \cap B = \{a\}\}
        \in \P(\Arch(A))
    \end{equation*}
    We claim that this set is in the image
    of the monomorphism
    $ \Free(\Arch(A)) \hookrightarrow \P(\Arch(A)) $
    described in \cref{eb431803ca}.
    We note that
    $ \Theta(\{a,b,a \vee b\},a) $
    is in the image for any $ b \in A $
    because an ideal~$ J $ belongs to this set
    if and only if $ J $ does not
    contain the ideal generated by
    $ (a \vee b)-a $.
    Hence
    $ \Theta(B,a) $ is also in the image
    since it can be written as $ \bigcap_{b \in B} \Theta(\{a,b,a \vee b\},a) $.

    For an object $ F \in \Fun(A, \cat{C})^{\omega} $,
    we can take a saturated subset~$ B $
    such that $ F $ is equivalent to the left Kan extension of $ F \rvert_{B} $.
    Then we define $ s(F) $,
    which is a priori dependent on~$ B $,
    as follows:
    \begin{equation*}
        s(F) = \bigvee_{a \in B}
        \sqrt{F(a)} \wedge \Theta(B,a)
        \in \Zar(\cat{C}^{\omega}) \otimes \Free(\Arch(A)).
    \end{equation*}
    Here we abuse the notation
    by identifying $ \Free(\Arch(A)) $ with its image
    under the monomorphism $ \Free(\Arch(A)) \hookrightarrow \P(\Arch(A)) $.

    First we wish to prove that $ s(F) $ is independent of the choice of~$ B $.
    Let $ B' $ be another saturated subset
    such that $ F $ is equivalent to the
    left Kan extension of $ F \rvert_{B'} $.
    We need to prove the following equality:
    \begin{equation*}
        \bigvee_{a \in B}
        \sqrt{F(a)} \wedge \Theta(B,a)
        =
        \bigvee_{a \in B'}
        \sqrt{F(a)} \wedge \Theta(B',a).
    \end{equation*}
    By considering a saturated set containing $ B \cup B' $,
    we may assume that $ B \subset B' $.
    For any minimal element~$ b' $ of $ B' \setminus B $,
    the subset $ B' \setminus \{b'\} $ is also saturated.
    Hence by induction
    we may also assume that $ B' \setminus B = \{b'\} $ for some $ b' \in B' $.
    If $ b' $ is a minimal element of $ B' $,
    then we get the equality since in this case
    $ F(b') $ is a zero object
    and $ \Theta(B,a) = \Theta(B',a) $ for any $ a \in B $.
    Let us consider the case when $ b' $ is not minimal.
    Since $ B $ is saturated,
    we can take the greatest element $ b $ of
    the set
    $ \{a \in B \mid a \leq b'\} $.
    Consider an element $ a \in B $.
    It is clear that $ \Theta(B,a) \supset \Theta(B',a) $ holds
    and this inclusion becomes an equality if $ a \nleq b' $.
    In fact,
    it also becomes an equality
    if $ a \leq b $ and $ a \neq b $:
    Suppose that $ J \in \Theta(B,a) $ fails
    to belong to $ \Theta(B',a) $.
    Then we have $ (a+J) \cap B' = \{a,b'\} $.
    But in this case, from the inequality
    $ 0 \leq b-a \leq b'-a $ we get $ b \in (a+J) \cap B $,
    which contradicts our assumption.
    Therefore,
    it is enough to show that
    $ \Theta(B,b) = \Theta(B',b) \cup \Theta(B',b') $ holds
    since we have $ F(b) \simeq F(b') $.
    First we prove $ \Theta(B,b) \subset \Theta(B',b) \cup \Theta(B',b') $.
    For $ J \in \Theta(B,b) \setminus \Theta(B',b) $,
    we have
    $ (b'+J) \cap B' \subset (b+J) \cap B' = \{b,b'\} $
    and so $ J \in \Theta(B',b') $.
    To prove the other inclusion,
    it remains to show that
    $ \Theta(B,b) \supset \Theta(B',b') $ holds.
    For $ J \in \Theta(B',b') $ and $ a \in (b+J) \cap B $,
    by $ 0 \leq (a \vee b') - b' = a - (a \wedge b') \leq a - b \in J $,
    we have $ a \vee b' \in (b'+J) \cap B' = \{b'\} $.
    This means $ a = b $,
    which completes the proof of the well-definedness of~$ s(F) $.

    Hence we have
    a function $ s \colon \Fun(A, \cat{C})^{\omega}
    \to \Zar(\cat{C}^{\omega}) \otimes \Free(\Arch(A)) $.
    We wish to prove that $ s $ is a support.
    Since it is an $ A $-semisupport and
    $ s(\unit) = 1 $ by definition,
    we are reduced to showing that $ s(F \otimes G) = s(F) \wedge s(G) $
    for $ F, G \in \Fun(A, \cat{C})^{\omega} $.
    We note that the assignments $ (F,G) \mapsto s(F \otimes G) $
    and $ (F,G) \mapsto s(F) \wedge s(G) $
    are both $ A $-semisupports in each variable.
    Hence by \cref{ss-gen},
    it suffices to show that
    $ s((C \otimes D)\{0/a_1\} \dotsb \{0/a_{m+n}\}) $
    equals to $ s(C\{0/a_1\} \dotsb \{0/a_m\})
    \wedge s(D\{0/a_{m+1}\} \dotsb \{0/a_{m+n}\}) $
    for $ C, D \in \cat{C}^{\omega} $
    and $ a_1, \dotsc, a_{m+n} \in A_{\geq 0} $.
    This claim follows
    if we have $ s(F\{0/b\}) = s(F) \wedge s(\unit\{0/b\}) $
    for $ F \in \Fun(A, \cat{C})^{\omega} $ and $ b \in A_{\geq 0} $.
    To prove this,
    by using \cref{ss-gen} again, we may assume
    that $ F $ can be written as $ C\{0/a_1\} \dotsb \{0/a_n\} $
    with $ C \in \cat{C}^{\omega} $ and $ a_1, \dotsc, a_n \in A_{\geq 0} $
    satisfying $ a_i \wedge a_j = 0 $ if $ i \neq j $.
    We may furthermore suppose
    that $ a_i > 0 $ for each~$ i $ and $ b > 0 $;
    otherwise the claim is trivial.
    For $ I \subset \{1, \dotsc, n\} $
    let $ a_I $ denote the sum $ \sum_{i \in I} a_i $,
    which is equal to
    the join $ \bigvee_{i \in I} a_i $ taken
    in~$ A_{\geq 0} $ by assumption.
    We now take the following two subsets of $ A $,
    both of which are saturated by assumption:
    \begin{align*}
        B &=
        \{a_I \mid I \subset \{1, \dotsc, n\}\},
        \\
        B' &=
        B \cup
        \{a_I + (a_{I'} \vee b) \mid \text{$ I, I' \subset \{1, \dotsc, n\} $
        satisfying $ I \cap I' = \emptyset $}\}.
    \end{align*}
    Then $ F $ and~$ F\{0/b\} $ are
    left Kan extensions of~$ F \rvert_B $ and~$ F\{0/b\} \rvert_{B'} $,
    respectively.
    We first prove that $ s(F\{0/b\}) \leq s(F) \wedge s(\unit\{0/b\}) $.
    Since we have $ s(F\{0/b\}) \leq s(F) $
    by the fact that $ s $ is an $ A $-semisupport,
    we need to show $ s(F\{0/b\}) \leq s(\unit\{0/b\}) $.
    Unwinding the definitions,
    we are reduced to proving that
    $ F\{0/b\}(c) \simeq 0 $
    or $ \Theta(B',c) \subset \Theta(\{0,b\}, 0) $ holds
    for each $ c \in B' $.
    If $ F\{0/b\}(c) $ is not zero,
    $ F(c) $ or $ F\{b\}(c) \simeq F(c-b) $ is not zero.
    These two cases are treated separately as follows:
    \begin{enumerate}
        \item
            If $ F(c) $ is not zero, then
            we have either $ c = 0 $ or $ c = b $.
            In the former case,
            we have indeed $ \Theta(B',0) \subset \Theta(\{0,b\},0) $
            since $ b \in B' $.
            In the latter case,
            we have that
            the morphism $ F(b) \to F(0) $ is
            equivalent to the identity of~$ C $.
            Hence $ F\{0/b\}(b) $ is zero, which is a contradiction.
        \item
            If $ F\{b\}(c) $ is not zero,
            then we have that $ c \in B $
            or $ c = a_I \vee b $ for some $ I \subset \{1, \dotsc, n\} $.
            In the former case,
            we have $ \Theta(B',c) \subset \Theta(\{0,b\},0) $
            since $ c+b \in B' $.
            In the latter case,
            we have $ c = a_I \vee b < a_I + b \leq c+b $;
            if the first inequality is an equality
            we have $ F\{b\}(c) = F\{b\}(a_I + b) \simeq F(a_I) \simeq 0 $,
            which is a contradiction.
            Combining this with $ a_I + b \in B' $,
            we have $ \Theta(B',c) \subset \Theta(\{0,b\},0) $.
    \end{enumerate}
    Next we prove
    that $ s(F\{0/b\}) \geq s(F) \wedge s(\unit\{0/b\}) $.
    Note that the right hand side equals to
    $ \sqrt{C} \wedge \{J \in \Theta(B,0) \mid b \notin J\} $
    by definition.
    Hence the claim follows
    from the observation
    that $ b \notin J \in \Theta(B,0) $
    implies $ J \in \Theta(B',0) $.

    Therefore we obtain
    a morphism~$ g \colon \Zar(\Fun(A, \cat{C})^{\omega})
    \to \Zar(\cat{C}^{\omega}) \otimes \Free(\Arch(A)) $
    of distributive lattices.
    We have that $ g \circ f $ is the identity
    by checking it for
    elements of $ \Zar(\cat{C}^{\omega}) $
    and $ \Arch(A_{\geq 0}) $.
    Also, by applying \cref{ss-gen},
    the computations $ g\bigl(\sqrt{C}\bigr) = \sqrt{C} $ for $ C \in \cat{C}^{\omega} $
    and $ g(\unit\{0/a\}) = \langle a \rangle $ for $ a \in A_{\geq 0} $ show
    that $ f \circ g $ is the identity.
    Hence we conclude that $ f $ is an isomorphism.
\end{proof}

\appendix

\section{Bases in higher topos theory}\label{s_basis}

In this appendix, we develop
the theory of bases for sites in the $ \infty $-categorical setting.
The main result is \cref{hyp-dense},
which says that an $ \infty $-site and its basis define
the same $ \infty $-topos after hypercompletion.
See \autocite{ultra} for a discussion in the $ 1 $-categorical setting.

The only result in this appendix
that we need in the main body of this paper
is \cref{ec5a5ba37f}, which is a corollary
of \cref{hyp-dense}.

\begin{remark}\label{3c155ff2f0}
    In the previous version of the preprint~\autocite{AsaiShah},
    Asai and Shah claimed
    that the map~$ i_* $ we describe in \cref{poset-case}
    is an equivalence, which implies \cref{ec5a5ba37f}.
    However, the claim is not true, as we see in \cref{31116799a2}.
    As this paper is being written,
    its corrected version, which
    still covers the case we use in the main body of this paper,
    has appeared. They use a different argument to prove it.
    We also note that we prove a conjecture stated in their preprint;
    see \cref{poset-case}.
\end{remark}

Since giving a Grothendieck topology on an $ \infty $-category
is the same thing as giving that on its homotopy category
(see \autocite[Remark~6.2.2.3]{HTT}),
we can translate notions used for sites
into the $ \infty $-categorical setting
without making any essential change.
Note that here we do not impose the existence of any pullbacks
on $ \infty $-categories.

\begin{definition}\label{b549dfa54c}
    An \emph{$ \infty $-site} is a small $ \infty $-category
    equipped with a Grothendieck topology.

    A \emph{basis} for an $ \infty $-site~$ \cat{C} $
    is a full subcategory $ \cat{B} \subset \cat{C} $ such that
    for every object $ C \in \cat{C} $ there exists
    a set of morphisms $ \{B_i \to C \mid i \in I\} $ that satisfies $ B_i \in \cat{B} $
    for all $ i \in I $ and generates a covering sieve on $ C $.
    Note that in this case there exists a unique
    Grothendieck topology on $ \cat{B} $ such that
    a sieve~$ \cat{B}_{/B}^{(0)} $ on~$ B $
    is covering if and only if its image under
    the inclusion $ \cat{B}_{/B} \hookrightarrow \cat{C}_{/B} $ generates a covering sieve.
    We always regard a basis as an $ \infty $-site by using
    this Grothendieck topology.
\end{definition}

\begin{example}\label{6f9ed17776}
    Considering the poset of open sets
    of a topological space
    equipped with the canonical topology,
    the notion of basis specializes
    to that of basis used in point-set topology.
\end{example}

\begin{remark}\label{eb4b669b9a}
    In ordinary topos theory,
    a basis is often referred to as a dense subsite.
    But in the $ \infty $-categorical setting,
    the term ``dense'' can be misleading since
    a basis need not define the same
    $ \infty $-topos; see \autocite[Example~20.4.0.1]{SAG}
    or \cref{31116799a2}.
\end{remark}

\begin{proposition}\label{7bf26b715a}
    Let $ \cat{C} $ be an $ \infty $-site.
    Suppose that $ \cat{B} $ is a basis for~$ \cat{C} $
    and $ F $ is a presheaf on~$ \cat{B} $.
    Then $ F $ is a sheaf on~$ \cat{B} $ if and only if
    its right Kan extension is a sheaf on~$ \cat{C} $.
    Especially, by restricting the right Kan extension functor
    $ \PShv(\cat{B}) \hookrightarrow \PShv(\cat{C}) $, we obtain
    a geometric embedding $ \Shv(\cat{B}) \hookrightarrow \Shv(\cat{C}) $.
\end{proposition}

We omit the proof of this fact
because this is proven in the same way
as in the $ 1 $-categorical setting;
see the second and third paragraphs of
the proof of \autocite[Proposition~B.6.6]{ultra},
but beware that some arguments
in the first paragraph cannot be translated to
our setting.

The main result is the following:

\begin{theorem}\label{hyp-dense}
    Let $ \cat{B} $ be a basis for an $ \infty $-site~$ \cat{C} $
    and $ G $ a presheaf on~$ \cat{C} $.
    Then $ G $ is a hypercomplete object of the $ \infty $-topos $ \Shv(\cat{C}) $
    if and only if the following conditions are satisfied:
    \begin{enumerate}
        \item
            \label{f4f8765973}
            The restriction~$ G \rvert_{\cat{B}^{\op}} $
            is a hypercomplete object of the $ \infty $-topos $ \Shv(\cat{B}) $.
        \item
            \label{76e550cb66}
            The functor~$ G $ is
            a right Kan extension of~$ G \rvert_{\cat{B}^{\op}} $.
    \end{enumerate}
\end{theorem}

We note that Jacob Lurie
let the author know
that this result could be proven using hypercoverings.
Here we will give a different proof,
which does not use (semi)simplicial machinery.

Before giving the proof of this theorem,
let us collect its formal consequences.

\begin{corollary}\label{778fc7ecc4}
    Let $ \cat{B} $ be a basis for an $ \infty $-site~$ \cat{C} $.
    Then the geometric embedding $ \Shv(\cat{B}) \hookrightarrow \Shv(\cat{C}) $
    obtained in \cref{7bf26b715a} is cotopological.
    In particular,
    it induces equivalences between their hypercompletions,
    their Postnikov completions,
    their bounded reflections,
    and their $ n $-localic reflections for any~$ n $.
\end{corollary}

\begin{corollary}\label{c3173bf985}
    Let $ \cat{B} $ be a basis for an $ \infty $-site~$ \cat{C} $.
    Suppose that both $ \cat{B} $ and~$ \cat{C} $ are $ n $-category
    for some~$ n $
    and have finite limits
    (but the inclusion need not preserve them).
    Then the geometric embedding obtained
    in \cref{7bf26b715a} is an equivalence.
\end{corollary}

\begin{remark}\label{daf8ad8f56}
    In \autocite[Lemma~C.3]{Hoyois14},
    Hoyois gave another sufficient condition
    under which the geometric embedding
    $ \Shv(\cat{B}) \hookrightarrow \Shv(\cat{C}) $ itself is an equivalence.
    As a special case of his result, if $ \cat{B} $ and~$ \cat{C} $ both admit
    pullbacks and the inclusion preserves them,
    it becomes an equivalence.
\end{remark}

Our proof uses
the following relative variant of \autocite[Lemma~20.4.5.4]{SAG}:

\begin{lemma}\label{rel-conn}
    Let $ f^* \colon \cat{Y} \to \cat{X} $ be the left adjoint of
    a geometric morphism between $ \infty $-toposes
    and $ \cat{D} \subset \cat{Y} $ an essentially small full subcategory.
    Suppose that for every $ Y \in \cat{Y} $ there
    exists a family of objects~$ (V_i)_{i \in I} $ of~$ \cat{D} $
    and a morphism $ \coprod_{i \in I} V_i \to Y $ whose
    image under~$ f^* $ is an effective epimorphism.
    Then the object $ f^* \bigl(\injlim_{V \in \cat{D}} V\bigr) \in \cat{X} $
    is $ \infty $-connective.
\end{lemma}

\begin{proof}
    This follows from a slight modification of
    the proof of \autocite[Lemma~20.4.5.4]{SAG}
    by using the fact that $ f^* $ preserves finite limits and colimits
    and it determines a functor $ \cat{Y}_{/Y} \to \cat{X}_{/f^*Y} $
    for any object $ Y \in \cat{Y} $,
    which is again the left adjoint of a geometric morphism.
\end{proof}

\begin{proof}[Proof of \cref{hyp-dense}]
    Let $ i^* \colon \PShv(\cat{C}) \to \PShv(\cat{B}) $ denote
    the restriction functor.
    We write $ i_* $ for its right adjoint.
    We let $ L $ be the sheafification
    functor associated to the $ \infty $-site~$ \cat{C} $
    and $ j \colon \cat{C} \to \PShv(\cat{C}) $
    the Yoneda embedding.

    Suppose that $ G $ is a hypercomplete object of $ \Shv(\cat{C}) $.
    We prove that the map $ G \to i_*i^*G $ is an equivalence.
    For $ C \in \cat{C} $,
    the map $ G(C) \to (i_*i^*G)(C) $ can be identified with
    the image of the morphism $ f \colon \injlim_{B \in \cat{B}_{/C}} j(B) \to j(C) $
    under the functor $ \Map_{\PShv(\cat{C})}(\X, G) $.
    Since $ G $ is a hypercomplete object of $ \Shv(\cat{C}) $,
    it suffices to prove that the morphism~$ Lf $ is $ \infty $-connective.
    We can see this by applying \cref{rel-conn}
    to the geometric embedding
    $ \Shv(\cat{C})_{/Lj(C)}
    \hookrightarrow \PShv(\cat{C})_{/j(C)}
    \simeq \PShv\bigl(\cat{C}_{/C}\bigr) $.

    Let $ \cat{X} $ denote the essential image of
    $ \Shv(\cat{C})^{\hyp} $ under~$ i^* $.
    It follows from what we have shown above
    that $ i^* $ restricts to determine
    an equivalence $ \Shv(\cat{C})^{\hyp} \to \cat{X} $.
    It also follows that the composition of left exact functors
    \begin{equation*}
        \PShv(\cat{B})
        \xrightarrow{i_*} \PShv(\cat{C})
        \xrightarrow{X \mapsto (LX)^{\hyp}} \Shv(\cat{C})^{\hyp}
        \xrightarrow{i^*} \cat{X}
    \end{equation*}
    is a left adjoint to the inclusion
    $ \mathcal{X} \hookrightarrow \PShv(\cat{B}) $:
    Indeed, for $ F \in \PShv(\cat{B}) $ and $ G \in \Shv(\cat{C})^{\hyp} $,
    we have
    \begin{equation*}
        \begin{split}
        \Map(F, i^*G)
        &\simeq
        \Map(i_*F, i_*i^*G)
        \simeq
        \Map(i_*F, G)
        \simeq
        \Map((Li_*F)^{\hyp}, G)
        \\
        &\simeq
        \Map((Li_*F)^{\hyp}, i_*i^*G)
        \simeq
        \Map(i^*((Li_*F)^{\hyp}), i^*G).
        \end{split}
    \end{equation*}
    Hence $ \cat{X} $ is a subtopos of $ \PShv(\cat{B}) $.
    Using \cref{7bf26b715a},
    we obtain inclusions
    $ \Shv(\cat{B})^{\hyp} \subset \cat{X} \subset \Shv(\cat{B}) $
    of subtoposes.
    Since $ \cat{X} \simeq \Shv(\cat{C})^{\hyp} $ is hypercomplete,
    the first inclusion is an equality.
    Therefore, the restriction of the adjoint pair
    $ (i^*, i_*) $ determines an equivalence
    $ \Shv(\cat{B})^{\hyp} \simeq \Shv(\cat{C})^{\hyp} $,
    which is a restatement of what we wanted to show.
\end{proof}

We conclude this appendix
by specializing
our results to the case
directly related to the main body of this paper.

\begin{example}\label{poset-case}
    Let $ P $ be a poset.
    Consider the canonical topology on the
    poset of open sets of~$ \Alex(P) $
    (see \cref{95b1e699de} for the definition).
    Then the subposet of principal cosieves,
    which is equivalent to~$ P^{\op} $,
    is a basis and the induced topology on it is trivial.
    Hence by \cref{778fc7ecc4},
    we have a cotopological inclusion
    $ i_* \colon \PShv(P^{\op}) \hookrightarrow \Shv(\Alex(P)) $.
    Since $ \PShv(P^{\op}) $ is already hypercomplete,
    $ i_* $ can be identified with the inclusion
    $ \Shv(\Alex(P))^{\hyp} \hookrightarrow \Shv(\Alex(P)) $.

    We note that this observation settles
    a conjecture posed by Asai and Shah in
    \autocite[Remark~2.6]{AsaiShah} affirmatively.
\end{example}

\begin{example}\label{ec5a5ba37f}
    In the situation of \cref{poset-case},
    the inclusion $ i_* $ is an equivalence
    when $ P $ is finite since
    in this case
    $ \Shv(\Alex(P)) $ is hypercomplete
    by \autocite[Corollaries~7.2.1.1 and~7.2.4.20]{HTT}.
    The same holds
    when $ P $ has finite joins since in this case
    $ \PShv(P^{\op}) $ is $ 0 $-localic.
    Actually, it is enough to assume
    that $ P $ has binary joins by \cref{daf8ad8f56}.
    We note that
    the class of all posets whose
    associated geometric embedding is an equivalence
    is closed under coproducts.
\end{example}

\begin{example}\label{31116799a2}
    In the situation of \cref{poset-case},
    the inclusion $ i_* $ itself need not be an equivalence
    in general.

    Consider the set $ P = \NN \times \{0,1\} $
    equipped with the ordering depicted as follows:
    \begin{equation*}
        \xymatrix{
            (0,1) \ar[r] \ar[rd] &
            (1,1) \ar[r] \ar[rd] &
            (2,1) \ar[r] \ar[rd] &
            (3,1) \ar[r] \ar[rd] &
            (4,1) \ar[r] \ar[rd] &
            \dotsb
            \\
            (0,0) \ar[r] \ar[ru] &
            (1,0) \ar[r] \ar[ru] &
            (2,0) \ar[r] \ar[ru] &
            (3,0) \ar[r] \ar[ru] &
            (4,0) \ar[r] \ar[ru] &
            \dotsb\rlap{.}
        }
    \end{equation*}
    We can check that the locale of open sets of $ \Alex(P) $
    is coherent, so
    the final object of $ \Shv(\Alex(P)) $ is compact
    by \autocite[Corollary~7.3.5.4]{HTT}.
    We now show that the final object of $ \PShv(P^{\op}) $,
    which is the constant functor taking the value $ * $,
    is not compact.
    If so,
    by \cref{02eb77190d},
    we can take $ n $ such that
    the final object is
    the left Kan extension of that of
    $ \Fun(\{0, \dotsc, n\} \times \{0,1\}, \Cat{S}) $,
    but then the value at $ (n+1, 0) $ becomes
    $ S^{n} $, which is a contradiction.

    We note that this $ \infty $-topos $ \Shv(\Alex(P)) $
    essentially appears in
    \autocite[Example~A9]{DuggerHollanderIsaksen04}.
    There they show the failure of hypercompleteness
    using a different argument.
\end{example}

\bibliographystyle{spmpsci}
\bibliography{references.bib}

\end{document}